\documentclass[12pt,reqno,english]{smfart}

\usepackage{amsmath,amssymb,smfthm,stmaryrd,epic,enumerate,dsfont}
\usepackage[english]{babel}
\usepackage[all]{xy}
\usepackage{csquotes}
\usepackage{mathabx}

\usepackage{xcolor}

\usepackage{mathrsfs}
\usepackage[active]{srcltx}

\newcommand*{\htarrow}{\lhook\joinrel\twoheadrightarrow}

\NumberTheoremsIn{subsection}\SwapTheoremNumbers

\begin{document}

\newcommand{\REMARK}[1]{\marginpar{\tiny #1}}
\newtheorem{thm}{Theorem}[subsection]
\newtheorem{lemma}[thm]{Lemma}
\newtheorem{corol}[thm]{Corollary}
\newtheorem{propo}[thm]{Proposition}
\newtheorem{defin}[thm]{Definition}
\newtheorem{Remark}[thm]{Remark}
\numberwithin{equation}{subsection}

\newtheorem{notas}[thm]{Notations}
\newtheorem{nota}[thm]{Notation}
\newtheorem{defis}[thm]{Definitions}
\newtheorem*{thm*}{Theorem}
\newtheorem*{prop*}{Proposition}
\newtheorem*{conj*}{Conjecture}

\def\Tm{{\mathbb T}}
\def\Um{{\mathbb U}}
\def\Am{{\mathbb A}}
\def\Fm{{\mathbb F}}
\def\Mm{{\mathbb M}}
\def\Nm{{\mathbb N}}
\def\Pm{{\mathbb P}}
\def\Qm{{\mathbb Q}}
\def\Zm{{\mathbb Z}}
\def\Dm{{\mathbb D}}
\def\Cm{{\mathbb C}}
\def\Rm{{\mathbb R}}
\def\Gm{{\mathbb G}}
\def\Lm{{\mathbb L}}
\def\Km{{\mathbb K}}
\def\Om{{\mathbb O}}
\def\Em{{\mathbb E}}
\def\Xm{{\mathbb X}}

\def\BC{{\mathcal B}}
\def\QC{{\mathcal Q}}
\def\TC{{\mathcal T}}
\def\ZC{{\mathcal Z}}
\def\AC{{\mathcal A}}
\def\CC{{\mathcal C}}
\def\DC{{\mathcal D}}
\def\EC{{\mathcal E}}
\def\FC{{\mathcal F}}
\def\GC{{\mathcal G}}
\def\HC{{\mathcal H}}
\def\IC{{\mathcal I}}
\def\JC{{\mathcal J}}
\def\KC{{\mathcal K}}
\def\LC{{\mathcal L}}
\def\MC{{\mathcal M}}
\def\NC{{\mathcal N}}
\def\OC{{\mathcal O}}
\def\PC{{\mathcal P}}
\def\UC{{\mathcal U}}
\def\VC{{\mathcal V}}
\def\XC{{\mathcal X}}
\def\SC{{\mathcal S}}
\def\RC{{\mathcal R}}

\def\BF{{\mathfrak B}}
\def\AF{{\mathfrak A}}
\def\GF{{\mathfrak G}}
\def\EF{{\mathfrak E}}
\def\CF{{\mathfrak C}}
\def\DF{{\mathfrak D}}
\def\JF{{\mathfrak J}}
\def\LF{{\mathfrak L}}
\def\MF{{\mathfrak M}}
\def\NF{{\mathfrak N}}
\def\XF{{\mathfrak X}}
\def\UF{{\mathfrak U}}
\def\KF{{\mathfrak K}}
\def\FF{{\mathfrak F}}

\def \longmapright#1{\smash{\mathop{\longrightarrow}\limits^{#1}}}
\def \mapright#1{\smash{\mathop{\rightarrow}\limits^{#1}}}
\def \lexp#1#2{\kern \scriptspace \vphantom{#2}^{#1}\kern-\scriptspace#2}
\def \linf#1#2{\kern \scriptspace \vphantom{#2}_{#1}\kern-\scriptspace#2}
\def \linexp#1#2#3 {\kern \scriptspace{#3}_{#1}^{#2} \kern-\scriptspace #3}

\def \Sel {{\mathop{\mathrm{Sel}}\nolimits}}
\def \Ext{\mathop{\mathrm{Ext}}\nolimits}
\def \ad{\mathop{\mathrm{ad}}\nolimits}
\def \sh{\mathop{\mathrm{Sh}}\nolimits}
\def \irr{\mathop{\mathrm{Irr}}\nolimits}
\def \FH{\mathop{\mathrm{FH}}\nolimits}
\def \FPH{\mathop{\mathrm{FPH}}\nolimits}
\def \coh{\mathop{\mathrm{Coh}}\nolimits}
\def \res{\mathop{\mathrm{Res}}\nolimits}
\def \op{\mathop{\mathrm{op}}\nolimits}
\def \rec {\mathop{\mathrm{rec}}\nolimits}
\def \art{\mathop{\mathrm{Art}}\nolimits}
\def \vol {\mathop{\mathrm{vol}}\nolimits}
\def \cusp {\mathop{\mathrm{Cusp}}\nolimits}
\def \scusp {\mathop{\mathrm{Scusp}}\nolimits}
\def \Iw {\mathop{\mathrm{Iw}}\nolimits}
\def \JL {\mathop{\mathrm{JL}}\nolimits}
\def \speh {\mathop{\mathrm{Speh}}\nolimits}
\def \isom {\mathop{\mathrm{Isom}}\nolimits}
\def \Vect {\mathop{\mathrm{Vect}}\nolimits}
\def \groth {\mathop{\mathrm{Groth}}\nolimits}
\def \hom {\mathop{\mathrm{Hom}}\nolimits}
\def \deg {\mathop{\mathrm{deg}}\nolimits}
\def \val {\mathop{\mathrm{val}}\nolimits}
\def \det {\mathop{\mathrm{det}}\nolimits}
\def \rep {\mathop{\mathrm{Rep}}\nolimits}
\def \spec {\mathop{\mathrm{Spec}}\nolimits}
\def \fr {\mathop{\mathrm{Fr}}\nolimits}
\def \frob {\mathop{\mathrm{Frob}}\nolimits}
\def \ker {\mathop{\mathrm{Ker}}\nolimits}
\def \im {\mathop{\mathrm{Im}}\nolimits}
\def \Red {\mathop{\mathrm{Red}}\nolimits}
\def \red {\mathop{\mathrm{red}}\nolimits}
\def \aut {\mathop{\mathrm{Aut}}\nolimits}
\def \diag {\mathop{\mathrm{diag}}\nolimits}
\def \spf {\mathop{\mathrm{Spf}}\nolimits}
\def \Def {\mathop{\mathrm{Def}}\nolimits}
\def \twist {\mathop{\mathrm{Twist}}\nolimits}
\def \supp {\mathop{\mathrm{Supp}}\nolimits}
\def \Id {{\mathop{\mathrm{Id}}\nolimits}}
\def \lie {{\mathop{\mathrm{Lie}}\nolimits}}
\def \Ind{\mathop{\mathrm{Ind}}\nolimits}
\def \ind {\mathop{\mathrm{ind}}\nolimits}
\def \bad {\mathop{\mathrm{Bad}}\nolimits}
\def \top {\mathop{\mathrm{Top}}\nolimits}
\def \ker {\mathop{\mathrm{Ker}}\nolimits}
\def \coker {\mathop{\mathrm{Coker}}\nolimits}
\def \gal {{\mathop{\mathrm{Gal}}\nolimits}}
\def \Nr {{\mathop{\mathrm{Nr}}\nolimits}}
\def \rn {{\mathop{\mathrm{rn}}\nolimits}}
\def \tr {{\mathop{\mathrm{Tr~}}\nolimits}}
\def \Sp {{\mathop{\mathrm{Sp}}\nolimits}}
\def \st {{\mathop{\mathrm{St}}\nolimits}}
\def \sp{{\mathop{\mathrm{Sp}}\nolimits}}
\def \perv{\mathop{\mathrm{Perv}}\nolimits}
\def \tor {{\mathop{\mathrm{Tor}}\nolimits}}
\def \gr {{\mathop{\mathrm{gr}}\nolimits}}
\def \nilp {{\mathop{\mathrm{Nilp}}\nolimits}}
\def \obj {{\mathop{\mathrm{Obj}}\nolimits}}
\def \spl {{\mathop{\mathrm{Spl}}\nolimits}}
\def \unr {{\mathop{\mathrm{Unr}}\nolimits}}
\def \alg {{\mathop{\mathrm{Alg}}\nolimits}}
\def \grr {{\mathop{\mathrm{grr}}\nolimits}}
\def \cogr {{\mathop{\mathrm{cogr}}\nolimits}}
\def \coFil {{\mathop{\mathrm{coFil}}\nolimits}}

\def \rem{{\noindent\textit{Remark.~}}}
\def \rems{{\noindent\textit{Remarques:~}}}
\def \ext {{\mathop{\mathrm{Ext}}\nolimits}}
\def \End {{\mathop{\mathrm{End}}\nolimits}}

\def\semi{\mathrel{>\!\!\!\triangleleft}}
\let \DS=\displaystyle
\def\HT{{\mathop{\mathcal{HT}}\nolimits}}

\def \hi{\HC}
\newcommand*{\tarrow}{\relbar\joinrel\mid\joinrel\twoheadrightarrow}
\newcommand*{\harrow}{\lhook\joinrel\relbar\joinrel\mid\joinrel\rightarrow}
\newcommand*{\rarrow}{\relbar\joinrel\mid\joinrel\rightarrow}
\def \coim {{\mathop{\mathrm{Coim}}\nolimits}}
\def \can {{\mathop{\mathrm{can}}\nolimits}}
\def\LFF{{\mathscr L}}

\setcounter{secnumdepth}{3} \setcounter{tocdepth}{3}

\def \Fil{\mathop{\mathrm{Fil}}\nolimits}
\def \CoFil{\mathop{\mathrm{CoFil}}\nolimits}
\def \Fill{\mathop{\mathrm{Fill}}\nolimits}
\def \CoFill{\mathop{\mathrm{CoFill}}\nolimits}
\def\SF{{\mathfrak S}}
\def\PF{{\mathfrak P}}
\def \EFil{\mathop{\mathrm{EFil}}\nolimits}
\def \EFill{\mathop{\mathrm{EFill}}\nolimits}
\def \FP{\mathop{\mathrm{FP}}\nolimits}

\let \longto=\longrightarrow
\let \oo=\infty

\let \d=\delta
\let \k=\kappa

\renewcommand{\theequation}{\arabic{section}.\arabic{subsection}.\arabic{thm}}
\newcommand{\marque}{\addtocounter{thm}{1}
{\smallskip \noindent \textit{\thethm}~---~}}

\renewcommand\atop[2]{\ensuremath{\genfrac..{0pt}{1}{#1}{#2}}}

\newcommand\atopp[2]{\genfrac{}{}{0pt}{}{#1}{#2}}

\title[Ihara's lemma in higher dimension]{Ihara's lemma and level rising in higher dimension}


\author{Boyer Pascal}
\email{boyer@math.univ-paris13.fr}
\address{Universit\'e Paris 13, Sorbonne Paris Cit\'e \\
LAGA, CNRS, UMR 7539\\ 
F-93430, Villetaneuse (France) \\
PerCoLaTor: ANR-14-CE25}

\frontmatter
%
%

\begin{abstract}
A key ingredient in the Taylor-Wiles proof of Fermat last theorem is the classical Ihara's lemma 
which is used to rise the modularity
property between some congruent galoisian representations. In their work on Sato-Tate, 
Clozel-Harris-Taylor proposed a generalization of the Ihara's lemma in higher dimension
for some similitude groups. The main aim of this paper is then to prove some new instances
of this generalized Ihara's lemma by considering some particular non pseudo Eisenstein maximal
ideals of unramified Hecke algebras. As a consequence, we prove a level rising statement.
\end{abstract}




\maketitle

\pagestyle{headings} \pagenumbering{arabic}

\tableofcontents
%
%

\section*{Introduction}
\renewcommand{\theequation}{\arabic{equation}}
\backmatter

Let $F=F^+E$ be a CM field with $E/\Qm$ quadratic imaginary.
For $\overline B/F$ a central division algebra with dimension $d^2$ equipped with a involution of
second specie $*$ and $\beta \in \overline B^{*=-1}$, consider the similitude group $\overline G/\Qm$ 
defined for any $\Qm$-algebra $R$ by
$$\overline G(R):= \bigl \{ (\lambda,g) \in R^\times \times (\overline B^{op} \otimes_\Qm R)^\times
\hbox{ such that } gg^{\sharp_\beta}=\lambda  \bigr \}$$
with $\overline B^{op}=\overline B \otimes_{F,c} F$ where $c=*_{|F}$ is the complex conjugation and
$\sharp_\beta$ the involution
$x \mapsto x^{\sharp_\beta}=\beta x^* \beta^{-1}$. For $p=uu^c$ decomposed in $E$, we have
$$\overline G(\Qm_p) \simeq \Qm_p^\times \times \prod_{w | u} (\overline B^{op}_v)^\times$$
where $w$ describes the places of $F$ above $u$. We suppose:
\begin{itemize}
\item the associated unitary group $\overline G_0(\Rm)$ being compact,

\item for any place $x$ of $\Qm$ inert or ramified in $E$, then $G(\Qm_x)$ is quasi-split.

\item There exists a place $v_0$ of $F$ above $u$ such that $\overline B_{v_0} \simeq D_{v_0,d}$
is the central division algebra over the completion $F_{v_0}$ of $F$ at $v_0$, with invariant $\frac{1}{d}$.
\end{itemize}
Fix a prime number $l \neq p$  and consider
a finite set $S$ of places of $F$ containing the ramification places $\bad$ of $\overline B$.
Let denote $\Tm_S/\overline \Zm_l$ the unramified Hecke algebra of $G$ outside $S$. 
For a cohomological minimal prime 
ideal $\widetilde{\mathfrak m}$ of $\Tm_S$, we can associate both a
near equivalence class of $\overline \Qm_l$-automorphic representation $\Pi_{\widetilde{\mathfrak m}}$ 
and a Galois representation
$$\rho_{\widetilde{\mathfrak m}}:G_F:=\gal(\bar F/F) \longrightarrow GL_d(\overline \Qm_l)$$ 
such that the eigenvalues of the Frobenius morphism at an unramified place $w$ are given by the 
Satake's parameter of the local component $\Pi_{\widetilde{\mathfrak m},w}$ of  
$\Pi_{\widetilde{\mathfrak m}}$.The semi-simple class of the reduction modulo $l$ of 
$\rho_{\widetilde{\mathfrak m}}$ depends only of the maximal
ideal $\mathfrak m$ of $\Tm$ containing $\widetilde{\mathfrak m}$.

\begin{conj*} \textbf{(Generalized Ihara's lemma by Clozel-Harris-Taylor)} \label{conj-ihara}
Consider
\begin{itemize}
\item an open compact subgroup $\overline U$ of $\overline G(\Am)$  such that outside $S$, its
local component is the maximal compact subgroup;

\item a place $w_0 \not \in S$ decomposed in $E$;

\item a maximal $\mathfrak m$ of $\Tm_S$ such that
$\overline \rho_{\mathfrak m}$ is absolutely irreducible. 
\end{itemize}
Let $\bar \pi$ be an irreducible sub-representation of 
$\mathcal C^\oo(\overline G(\Qm) \backslash \overline G(\Am)/ \overline U^{w_0},
\overline \Fm_l)_{\mathfrak m}$, where $\overline U=\overline U_{w_0} \overline U^{w_0}$, 
then its local component $\bar \pi_{w_0}$ at $w_0$ is generic.
\end{conj*}

\rem In its classical version for $GL_2$, Ihara's lemma is used to rise the modularity property between
some congruent galoisian representations and so was the role of this higher dimensional version in
Clozel-Harris-Taylor paper on the Sato-Tate conjecture. 
Shortly after Taylor found an argument to avoid Ihara's lemma.
However this conjecture remains highly interesting, see for example works of Clozel-Thorne \cite{clo-thor}, or 
Emerton-Helm \cite{emer-helm}. 

The main result of this paper is the following instances of the previous conjecture.

\begin{thm*} 
The previous generalized Ihara's conjecture is true if the maximal ideal $\mathfrak m$ verifies the 
following extra properties.
\begin{itemize}

\item[(H1)] $\overline \rho_{\mathfrak m}$ is absolutely irreducible and $l \geq d+1$ except\footnote{cf. the main result of \cite{boyer-bordeaux}} 
maybe for a finite number of exceptions depending of the extension $F/\Qm$.

\item[(H2)] The image of $\overline \rho_{\mathfrak m,w_0}$ in its Grothendieck group
is multiplicity free and\footnote{Using the main result of \cite{boyer-local-ihara} we could take off
the condition about not containing a full Zelevinsky line, 
cf. the last remark of \S \ref{para-proof}.} does not contain any full Zelevinsky line.\footnote{In 
particular $q_{w_1}$ can't  be congruent to $1$ modulo $l$.}

\item[(H3)] The image of $\overline \rho_{\mathfrak m,v_0}$ which by Jacquet-Langlands 
correspondence is associated to some superSpeh, $\speh_s(\varrho_{v_0})$, 
for some supercuspidal
$\overline \Fm_l$-representation $\varrho_{v_0}$ of $GL_g(F_{v_0}$ with $d=sg$,
cf. theorem 3.1.4 of \cite{dat-jl}. We then suppose that the set, cf. the notation in \S \ref{para-jl}
$$\bigl \{ \varrho_{v_0},\varrho_{v_0} \{ 1 \}, \cdots ,\varrho_{v_0} \{ s-1 \} \bigr \}$$ 
is of cardinal $s$.

\end{itemize}
\end{thm*}

\rem Concerning (H1), you can replace it by any hypothesis which insures the freeness of the 
cohomology of the Kottwitz-Harris-Taylor Shimura variety of \S \ref{para-shimura}. In \cite{boyer-imj}
we gave two such conditions, the one given by (H1) is proved in \cite{boyer-bordeaux}. We give
more details before theorem \ref{theo2}.

\rem About (H2) note that the first condition also appears in section 4.5 of \cite{CHT}
in the statements of level raising. Concerning the second condition of (H2), using the main
result of \cite{boyer-local-ihara}, we can remove it, cf. the remark after the proof of 
lemma \ref{lem-ss}.

\medskip

To prove this result, we first translate such a property to 
the cohomology group of middle degree of the Kottwitz-Harris-Taylor Shimura variety
$X_{\IC}$ associated to the similitude group $G/\Qm$ such that
\begin{itemize}
\item $G(\Am^{\oo})=\overline G(\Am^{\oo,p}) \times GL_d(F_{v_0}) \times
\prod_{\atop{w | u}{w \neq v_0}} (\overline B_{w}^{op})^\times$,

\item the signatures of $G(\Rm)$ are $(1,d-1) \times (0,d) \times \cdots \times (0,d)$.
\end{itemize}
In particular to each prime ideal $\widetilde{\mathfrak m}$ of $\Tm_S$, 
is associated a $\overline \Qm_l$-irreducible
automorphic representation $\Pi_{\widetilde{\mathfrak m}}$ of $G(\Am_\Qm)$ whose Satake
parameters at finite places outside $S$ are prescribed by $\widetilde{\mathfrak m}$.
We then compute the cohomology groups of the geometric generic fiber of $X_\IC$
through the spectral sequence of vanishing cycles at the place $v_0$.
Thanks to (H1), we know, cf. \cite{boyer-bordeaux}, 
the $H^{i}(X_U,\overline \Zm_l)_{\mathfrak m}$ to be
free and so, $H^{i}(X_U,\overline \Fm_l)_{\mathfrak m}=(0)$ for $i \neq d-1$.

\medskip

\rem Moreover (H2) (resp. (H3)) insures that the graduates of the filtration of
$H^{d-1}(X_U,\overline \Fm_l)_{\mathfrak m}$, given by the entire version of the weight-monodromy
filtration, at the place $w_0$ (resp. $v_0$) are also free.

\medskip

The contribution of the supersingular points of the special fiber at $v_0$, using (H3),
allows us to associate to
an irreducible sub-representation $\overline \pi$ of 
$\mathcal C^\oo(\overline G(\Qm) \backslash \overline G(\Am)/ \overline U^{w_0},
\overline \Fm_l)_{\mathfrak m}$, an irreducible sub-representation $\pi$ of
$H^{d-1}(X_U,\overline \Fm_l)_{\mathfrak m}$, such that $\pi^{\oo,v_0} \simeq \bar \pi^{\oo,v_0}$. 
We then try to prove the genericness of $\pi_{w_0}$  by proving, using (H2), such genericness
property of irreducible submodules of $H^{d-1}(X_U,\overline \Fm_l)_{\mathfrak m}$.
On ingredient in \S \ref{para-reseau}, comes from \cite{scholze-LT} \S 5, 
where the hypothesis that $\overline \rho_{\mathfrak m}$ is absolutely 
irreducible, insures that the lattices of isotypic components of
$H^{d-1}(X_U,\overline \Qm_l)_{\mathfrak m}$ given by the $\overline \Zm_l$-cohomology,
can be written as a tensorial product of stable lattices for the $G(\Am^{\oo})$ and the Galois actions.

Finally (H2) is needed to control the combinatorics.

\medskip

\rem As pointed out to us by M. Harris, the case where the cardinal
$q_{w_0}$ of the residue field at $w_0$, is congruent to $1$ modulo $l$ should be of crucial
importance for the application. Meanwhile our strategy relies on the construction of
a filtration of $H^{d-1}(X_U,\overline \Fm_l)_{\mathfrak m}$ which each graduates verify
the genericness property of irreducible submodule and where these graduates are parabolically
induced. When $q_{w_0} \equiv 1 \mod l$, parabolically
induced $\overline \Fm_l$-representations are often semi-simple and so they can't verify such
genericness property of irreducible submodule. It seems that our approach is not well 
adapted to treat this fundamental case.

\medskip

To state our application to level rising, denote $\SC_{w_0}(\mathfrak m)$ the supercuspidal
support of the modulo $l$ reduction of $\Pi_{\widetilde{\mathfrak m},w_0}$ for any prime ideal
$\widetilde{\mathfrak m} \subset \mathfrak m$: it depends only on $\mathfrak m$. By (H2)
this support is multiplicity free and we first break it
$\SC_{w_0}(\mathfrak m)=\coprod_{\varrho} \SC_{\varrho}(\mathfrak m)$
according to the inertial equivalence classes $\varrho$ of irreducible $\overline \Fm_l$-representations
of some $GL_{g(\varrho)}(F_{w_0})$ with $1 \leq g(\varrho) \leq d$. For any such $\varrho$ we then denote 
$l_1(\varrho) \geq \cdots \geq l_{r(\varrho)}(\varrho) \geq 1$,
such that $S_{\varrho}(\mathfrak m)$ can be written as the union of $r(\varrho)$ Zelevinsky 
unlinked segments of length $l_i(\varrho)$
$$[\varrho \nu^k_i,\bar \rho \nu^{k+l_i(\varrho)}]=\bigl \{ \varrho \nu^k,\varrho \nu^{k+1},\cdots,
\varrho \nu^{k+l_i(\varrho)} \bigr \}.$$
Then for any minimal prime ideal
$\widetilde{\mathfrak m} \subset \mathfrak m$, and $\Pi \in \Pi_{\widetilde{\mathfrak m}}$,
we write its local component
$\Pi_{w_0} \simeq \bigtimes_{\varrho} \Pi_{w_0}(\varrho)$
and $\Pi_{w_0}(\varrho) \simeq \bigtimes_{i=1}^{r(\varrho)} \Pi_{w_0}(\varrho,i)$
where for each $1 \leq i \leq r(\varrho)$, the modulo $l$ reduction of the supercuspidal support
of $\Pi_{w_0}(\varrho,i)$ is, with the notations of the previous section, those of the Zelevinsky segment 
$[\varrho \nu^{\delta_i},\varrho \nu^{\delta_i+l_i(\varrho)}]$.

\begin{prop*} 
Take a maximal ideal $\mathfrak m$ verifying the hypothesis (H1) and (H2).
Let $\varrho_0$ such that $\SC_{\varrho_0}(\mathfrak m)$ is non empty and consider 
$1 \leq i \leq r(\varrho_0)$. Then there exists a minimal prime ideal 
$\widetilde{\mathfrak m} \subset \mathfrak m$ and an automorphic
representation $\Pi \in \Pi_{\widetilde{\mathfrak m}}$ such that with the previous notation
$\Pi_{w_0}(\varrho_0,i)$ is non degenerate, i.e. isomorphic to $\st_{l_i(\varrho)}(\pi_{w_0})$
for some irreducible cuspidal $\overline \Qm_l$-representation $\pi_{w_0}$.
\end{prop*}

In particular when there is only one segment, which is always the case for $GL_2$, then the result is optimal.

\medskip

\rem In the previous proposition, we could also prove that for any such $\widetilde{\mathfrak m}$
and any $\Pi \in \Pi_{\widetilde{\mathfrak m}}$, then $\Pi_{w_0}(\varrho_0,i)$ is non degenerate,
which looks similar to theorem 2.1 of \cite{allen-newton} where $\overline \rho_{\mathfrak m}$ 
is supposed to be absolutely irreducible and decomposed generic which also imposes
the cohomology groups to be free.

\mainmatter

\renewcommand{\theequation}{\arabic{section}.\arabic{subsection}.\arabic{smfthm}}

\section{Shimura variety of Kottwitz-Harris-Taylor type}

\subsection{Geometry}
\label{para-shimura}

Recall from the introduction that a prime number $l$ is fixed, distinct from all others prime
numbers which will be considered in the following. Let $F=F^+E$ be a CM field with $E/\Qm$
imaginary quadratic such that $l$ is unramified, and $F^+$ totally real with a fixed embedding 
$\tau:F^+ \hookrightarrow \Rm$. For a place $v$ of $F$, we denote $F_v$ the completion
of $F$ at $v$, with ring of integers $\OC_v$, uniformizer $\varpi_v$ and residual field
$\kappa(v)$ with cardinal $q_v$.

Let $B$ be a central division algebra over $F$ of dimension $d^2$ such that at any place $x$
of $F$, either $B_x$ is split either it's a division algebra. We moreover suppose the existence of
an involution of second specie $*$ on $B$ such that $*_{|F}$ is the complex conjugation $c$.
For $\beta \in B^{*=-1}$, we denote $\sharp_\beta:x \mapsto \beta x^* \beta^{-1}$ and let $G/\Qm$
such that for any $\Qm$-algebra $R$, 
$$G(R) = \bigl \{ (\lambda,g) \in R^\times \times (B^{op} \otimes_\Qm R)^\times \hbox{ such that }
gg^{\sharp_\beta}=\lambda \bigr \},$$
with $B^{op}=B \otimes_{c} F$. If $x=yy^c$ is split in $E$ then
$$G(\Qm_x) \simeq (B^{op}_y)^\times \times Q_x^\times \simeq \Qm_x^\times \times \prod_{z_i}
(B_{z_i}^{op})^\times$$
where, identifying the places of $F^+$ above $x$ with those of $F$ above $y$, we write $x=\prod_i z_i$.
Moreover we can impose that
\begin{itemize}
\item if $x$ is inert in $E$ then $G(\Qm_x)$ is quasi-split,

\item the signature of $G(\Rm)$ are $(1,d-1) \times (0,d) \times \cdots \times (0,d)$.
\end{itemize}
With the notations of the introduction we have
$$G(\Am^{\oo})=\overline G(\Am^{\oo,p}) \times \Bigl ( \Qm_{p_{v_0}}^\times 
GL_d(F_{v_0}) \times \prod_{\atop{w | u}{w \neq v_0}} 
(\overline B_{w}^{op})^\times \Bigr ).$$

\begin{defi}
We denote $\bad$ the set of places $w$ of $F$ such that $B_w$ is non split. Let $\spl$
the set $w$ of finite places of $F$ not in $\bad$ such that $w_{|\Qm}$ is split in $E$. For such
a place $w$ with $p=w_{|\Qm}$, we write abusively
$$G(\Am^{w})=G(\Am^{p}) \times \Qm_p^\times \times \prod_{\atop{u|p}{u \neq w}} (B_u^{op})^\times,$$
and $G(F_w)=GL_d(F_w)$.
\end{defi}

\rem With the notations of the introduction, the role of $w$ in the previous definition will be taken by
$v_0$, $v_1$ or $w_0$.

\begin{nota}
For all open compact subgroup $U^p$ of $G(\Am^{\oo,p})$ and 
$m=(m_1,\cdots,m_r) \in \Zm_{\geq 0}^r$, we consider
$$U^p(m)=U^p \times \Zm_p^\times \times \prod_{i=1}^r \ker ( \OC_{B_{v_i}}^\times 
\longrightarrow (\OC_{B_{v_i}}/\PC_{v_i}^{m_i})^\times ).$$
For $w_0$ one of the $v_i$ and $n \in \Nm$, we also introduce $U^{w_0}(n):=U^p(0,\cdots,0,n,0,\cdots,0)$.
\end{nota}

We then denote $\IC$ for the set of these $U^p(m)$ such that there exists a place
$x$ for which the projection from $U^p$ to $G(\Qm_x)$ doesn't contain any element
with finite order except the identity, cf. \cite{h-t} bellow of page 90. 

Attached to each $I \in \IC$ is a Shimura variety $X_I \rightarrow \spec \OC_v$ 
of type Kottwitz-Harris-Taylor and we denote $X_\IC=(X_I)_{I \in \IC}$ the projective
system: recall that the transition morphisms
$r_{J,I}:X_J \rightarrow X_I$ are finite flat and even etale when $m_1(J)=m_1(I)$.
This projective system is then equipped with a Hecke action of $G(\Am^\oo) \times \Zm$, 
where the action of $z$ in the Weil group $W_v$ of $F_v$ is given par
$-\deg(z) \in \Zm$ where $\deg=\val \circ \art_v^{-1}$, where $\art_v^{-1}:W_v^{ab} \simeq F_v^\times$
is the Artin isomorphism which sends geometric Frobenius to uniformizers.

\begin{notas} (cf. \cite{boyer-invent2} \S 1.3) \label{nota-strate}
Let $I \in \IC$, 
\begin{itemize}
\item the special fiber of $X_I$ will be denoted $X_{I,s}$ and its geometric
special fiber $X_{I,\bar s}:=X_{I,s} \times \spec \overline \Fm_p$.

\item For $1 \leq h \leq d$, let $X_{I,\bar s}^{\geq h}$ (resp. $X_{I,\bar s}^{=h}$)
be the closed (resp. open) Newton stratum of height $h$, defined as the subscheme
where the connected component of the universal Barsotti-Tate group is of rank 
greater or equal to $h$ (resp. equal to $h$).
\end{itemize}
\end{notas}

\noindent \textit{Remark}: $X_{\IC,\bar s}^{\geq h}$ is of pure dimension $d-h$.
For $1 \leq h < d$, the Newton stratum $X_{\IC,\bar s}^{=h}$ is geometrically induced
under the action of the parabolic subgroup $P_{h,d-h}(F_v)$, defined as the
stabilizer of the first $h$ vectors of the canonical basis of $F_v^d$. Concretly this means there
exists a closed subscheme $X_{I,\bar s,\overline{1_h}}^{=h}$ stabilized by the Hecke 
action of $P_{h,d-h}(F_v)$ and such that
$$X_{\IC,\bar s}^{=h} \simeq X_{\IC,\bar s,\overline{1_h}}^{=h} 
\times_{P_{h,d-h}(F_v)} GL_d(F_v).$$

\begin{notas} \label{nota-j}
Let denote
$$i^h:X^{\geq h}_{\IC,\bar s} \hookrightarrow X^{\geq 1}_{\IC,\bar s}, \quad
j^{\geq h}: X^{=h}_{\IC,\bar s} \hookrightarrow X^{\geq h}_{\IC,\bar s}$$
and $j^{=h}=i^h j^{\geq h}$.
\end{notas}

\subsection{Jacquet-Langlands correspondence and $\varrho$-type}
\label{para-jl}

For a representation $\pi_v$ of $GL_d(F_v)$ and $n \in \frac{1}{2} \Zm$, set 
$\pi_v \{ n \}:= \pi_v \otimes q_v^{-n \val \circ \det}$. Recall that
the normalized induction of two representations $\pi_{v,1}$ and $\pi_{v,2}$ 
of respectively $GL_{n_1}(F_v)$ and $GL_{n_2}(F_v)$ is
$$\pi_1 \times \pi_2:=\ind_{P_{n_1,n_1+n_2}(F_v)}^{GL_{n_1+n_2}(F_v)}
\pi_{v,1} \{ \frac{n_2}{2} \} \otimes \pi_{v,2} \{-\frac{n_1}{2} \}.$$
A representation
$\pi_v$ of $GL_d(F_v)$ is called \emph{cuspidal} (resp. \emph{supercuspidal})
if it's not a subspace (resp. subquotient) of a proper parabolic induced representation.
When the field of coefficients is of characteristic zero then these two notions coincides,
but this is no more true for $\overline \Fm_l$.

\rem The modulo $l$ reduction of an irreducible $\overline \Qm_l$-representation is still irreducible
and cuspidal but not necessary supercuspidal. In this last case, its supercuspidal support is a
Zelevinsky segment associated to some unique inertial equivalent classe $\varrho$, where
$\varrho$ is an irreducible supercuspidal $\overline \Fm_l$-representation. Thanks to (H2) we will
not be concerned by this subtility.

\begin{defi} We say that $\pi_v$ is of type $\varrho$ when the supercuspidal support of its modulo
$l$ reduction is contained in the Zelevinsky line of $\varrho$.
\end{defi}

\begin{defin} \label{defi-rep} (see \cite{zelevinski2} \S 9 and \cite{boyer-compositio} \S 1.4)
Let $g$ be a divisor of $d=sg$ and $\pi_v$ an irreducible cuspidal 
$\overline \Qm_l$-representation of $GL_g(F_v)$. Then the normalized induced representation 
$$\pi_v\{ \frac{1-s}{2} \} \times \pi_v \{ \frac{3-s}{2} \} \times \cdots \times \pi_v \{ \frac{s-1}{2} \}$$ 
holds a unique irreducible quotient (resp. subspace) denoted $\st_s(\pi_v)$ (resp.
$\speh_s(\pi_v)$); it's a generalized Steinberg (resp. Speh) representation.
\end{defin}

\rem If $\chi_v$ is a character of $F_v^\times$ then $\speh_s(\chi_v)=\chi_v \circ \det$.

The local Jacquet-Langlands correspondance is a bijection between irreducible essentially square
integrable representations of $GL_d(F_v)$, i.e. representations of the type $\st_s(\pi_v)$
for $\pi_v$ cuspidal, and irreducible representations of $D_{v,d}^\times$ where
$D_{v,d}$ is the central division algebra over $F_v$ with invariant $\frac{1}{d}$. 

\begin{nota}
We will denote $\pi_v[s]_D$ the irreductible representation of $D_{v,d}^\times$ associated 
to $\st_s(\pi_v^\vee)$ by the local Jacquet-Langlands correspondance.
\end{nota}

We denote $\RC_{ \overline \Fm_l}(d)$ the set of equivalence classes of irreducible
$\overline \Fm_l$-representations of $D_{v,d}^\times$. For $\bar \tau \in \RC_{ \overline \Fm_l}(d)$,
let $\CC_{\bar \tau}$ be the sub-category of smooth $\Zm_l^{nr}$-representations of 
$D_{v,d}^\times$ with objects whose irreducible sub-quotients are isomorphic
to a sub-quotient of $\bar \tau_{|\DC_{v,d}^\times}$. Note that
$\CC_{\bar \tau}$ is a direct factor inside $\rep_{\Zm_l^{nr}}^\oo(D_{v,d}^\times)$ so that
every $\Zm_l^{nr}$-representation $V_{\Zm_l^{nr}}$ of $D_{v,d}^\times$
can be decomposed as a direct sum
$$V_{\Zm_l^{nr}} \simeq \bigoplus_{\bar \tau \in \RC_{ \overline \Fm_l}(d)} V_{\Zm_{l,\bar \tau}^{nr}}$$
where $V_{\Zm_{l,\bar \tau}^{nr}}$ is an object of $\CC_{\bar \tau}$.

Let $\pi_v$ be an irreducible cuspidal representation of $GL_g(F_v)$ and fix an integer
$s \geq 1$. Then the modulo $l$ reduction of $\speh_{s}(\pi)$ is irreducible,  cf. \cite{dat-jl} \S 2.2.3. 

\begin{nota}
When the modulo $l$ reduction of $\pi$, denoted $\varrho$, is supercuspidal,
then we will denote $\speh_s(\varrho)$ the modulo $l$ reduction of $\speh_s(\pi)$:
we call it a $\overline \Fm_l$- superSpeh representation. 
\end{nota}

By \cite{dat-jl} 3.1.4, we have a bijection
\begin{multline} \label{eq-bij-type}
\Bigl \{ \bar \Fm_{l}-\hbox{superspeh irreducible representations of  }GL_{d}(F_v) \Bigr \} \\ \simeq \\
\Bigl \{ \bar \Fm_{l}-\hbox{representations irreducible of } D^{\times}_{v,d} \Bigr \} 
\end{multline}
compatible with the modulo $l$ reduction in the sense that if $\pi_v$ is a lifting of $\varrho$,
then the modulo $l$ reduction of $\pi^\vee[s]_{D}$ matches through the previous bijection, with 
the superSpeh $\speh_{s}(\varrho)$.

\begin{defi} \label{defi-rho-type}
A $\overline \Fm_l$-representation of $D^\times_{v,d}$ (resp. an irreducible cuspidal
representation of $GL_d(F_v)$) is said of type $\varrho$ if
all its irreducible sub-quotient are, through the previous bijection, associated to some superSpeh 
$\speh_s(\varrho)$ (resp. its supercuspidal support belongs to the Zelevinsky line of $\varrho$).
\end{defi}

Recall that if $\epsilon(\varrho)$ is the cardinal of the Zelevinsky line associated to $\varrho$ and, 
cf. \cite{vigneras-induced} p.51, then let
$$m(\varrho)=\left\{ \begin{array}{ll} \epsilon(\varrho), & \hbox{si } \epsilon(\varrho)>1; \\ l, & \hbox{sinon.} \end{array} \right.$$

\begin{nota}
Let $r(\varrho)$ the biggest integer $i$ such that $l^i$ divides $\frac{d}{m(\varrho)g}$ and
define 
$$g_{-1}(\varrho)=g \quad \hbox{ and } \quad \forall 0 \leq i \leq r(\varrho),~
g_i(\varrho)=m(\varrho)l^i g.$$
We also denote $s_i(\varrho):=\lfloor \frac{d}{g_i(\varrho)} \rfloor$.
\end{nota}

Then if $\pi_v$ is an irreducible cuspidal $\overline \Qm_l$-representation of $GL_k(F_v)$ with
type $\varrho$, then there exists $i$ such that $k=g_i$. We say that $\pi_v$ is of $\varrho$-type $i$
and we denote $\scusp_i(\varrho)$ the set of inertial equivalence classes of irreducible
cuspidal representations of $\varrho$-type $i$.

\begin{nota}
For $\varrho$ an irreducible supercuspidal $\overline \Fm_l$- representation of $GL_g(F_v)$,
we denote $\RC_\varrho=\coprod_{s \geq 1} \RC_\varrho(sg)$ where $\RC_\varrho(sg)$
is the set of equivalence classes of irreducible $\overline \Fm_l$-representations
of $D_{v,sg}^\times$ of type $\varrho$.
\end{nota}

\subsection{Harris-Taylor local systems}

Let $\pi_v$ be an irreducible cuspidal $\overline \Qm_l$-representation of $GL_g(F_v)$ and
fix $t \geq 1$ such that $tg \leq d$. Thanks to Igusa varieties, Harris and Taylor
constructed a local system on  $X^{=tg}_{\IC,\bar s,\overline{1_h}}$
$$\LC_{\overline \Qm_l}(\pi_v[t]_D)_{\overline{1_h}}=\bigoplus_{i=1}^{e_{\pi_v}} 
\LC_{\overline \Qm_l}(\rho_{v,i})_{\overline{1_h}}$$
where $(\pi_v[t]_D)_{|\DC_{v,h}^\times}=\bigoplus_{i=1}^{e_{\pi_v}} \rho_{v,i}$ 
with $\rho_{v,i}$ irreductible. The Hecke action of $P_{tg,d-tg}(F_v)$ is then given
through its quotient $GL_{d-tg} \times \Zm$. These local systems have stable
$\overline \Zm_l$-lattices and we will write simply $\LC(\pi_v[t]_D)_{\overline{1_h}}$
for any $\overline \Zm_l$-stable lattice that we don't want to specify.

\begin{notas} For $\Pi_t$ any representation of $GL_{tg}$ and 
$\Xi:\frac{1}{2} \Zm \longrightarrow \overline \Zm_l^\times$ defined by 
$\Xi(\frac{1}{2})=q^{1/2}$, we introduce
$$\widetilde{HT}_1(\pi_v,\Pi_t):=\LC(\pi_v[t]_D)_{\overline{1_h}} 
\otimes \Pi_t \otimes \Xi^{\frac{tg-d}{2}}$$
and its induced version
$$\widetilde{HT}(\pi_v,\Pi_t):=\Bigl ( \LC(\pi_v[t]_D)_{\overline{1_h}} 
\otimes \Pi_t \otimes \Xi^{\frac{tg-d}{2}} \Bigr) \times_{P_{tg,d-tg}(F_v)} GL_d(F_v),$$
where the unipotent radical of $P_{tg,d-tg}(F_v)$ acts trivially and the action of
$$(g^{\oo,v},\left ( \begin{array}{cc} g_v^c & *†\\ 0 & g_v^{et} \end{array} \right ),\sigma_v) 
\in G(\Am^{\oo,v}) \times P_{tg,d-tg}(F_v) \times W_v$$ 
is given
\begin{itemize}
\item by the action of $g_v^c$ on $\Pi_t$ and 
$\deg(\sigma_v) \in \Zm$ on $ \Xi^{\frac{tg-d}{2}}$, and

\item the action of $(g^{\oo,v},g_v^{et},\val(\det g_v^c)-\deg \sigma_v)
\in G(\Am^{\oo,v}) \times GL_{d-tg}(F_v) \times \Zm$ on $\LC_{\overline \Qm_l}
(\pi_v[t]_D)_{\overline{1_h}} \otimes \Xi^{\frac{tg-d}{2}}$.
\end{itemize}
We also introduce
$$HT(\pi_v,\Pi_t)_{\overline{1_h}}:=\widetilde{HT}(\pi_v,\Pi_t)_{\overline{1_h}}[d-tg],$$
and the perverse sheaf
$$P(t,\pi_v)_{\overline{1_h}}:=j^{=tg}_{\overline{1_h},!*} HT(\pi_v,\st_t(\pi_v))_{\overline{1_h}} 
\otimes \Lm(\pi_v),$$
and their induced version, $HT(\pi_v,\Pi_t)$ and $P(t,\pi_v)$, where 
$$j={h}=i^h \circ j^{\geq h}:X^{=h}_{\IC,\bar s} \hookrightarrow
X^{\geq h}_{\IC,\bar s} \hookrightarrow X_{\IC,\bar s}$$ 
and $\Lm^\vee$ is the local Langlands correspondence.
\end{notas}

\rem Recall that $\pi'_v$ is said inertially equivalent to $\pi_v$ 
if there exists a character $\zeta: \Zm \longrightarrow  \overline \Qm_l^\times$ such that
$\pi'_v \simeq \pi_v \otimes (\zeta \circ \val \circ \det)$.
Note, cf. \cite{boyer-invent2} 2.1.4, that $P(t,\pi_v)$ depends only on the inertial class of $\pi_v$ and
$$P(t,\pi_v)=e_{\pi_v} \PC(t,\pi_v)$$ 
where $\PC(t,\pi_v)$ is an irreducible perverse sheaf. When we want to speak of
the $\overline \Qm_l$-versions we will add it on the notations.

\begin{defi} We will say that $HT(\pi_v,\Pi_t)$ or $\PC(t,\pi_v)$ is of type $\varrho$
if $\pi_v$ is.
\end{defi}

\begin{lemm} \label{lem-coho-flou}
If $\rho \otimes \sigma$ is a $GL_d(F_v) \times W_v$ equivariant irreducible sub-quotient
of $H^i(X_{\IC,\bar s_v},\PC(\pi_v,t) \otimes_{\overline \Zm_l} \overline \Fm_l)$ then
\begin{itemize}
\item $\sigma$ is an irreducible sub-quotient of the modulo $l$ reduction of $\Lm(\pi_v \otimes \chi_v)$,
where $\chi_v$ is an uramified character of $F_v$, and

\item $\rho$ is an irreducible sub-quotient of the modulo $l$ reduction of a induced representation
of the following shape $\st_t(\pi_v \otimes \chi_v) \times \psi_v$ where $\psi_v$ 
is an entire irreducible representation of $GL_{d-tg}(F_v)$.
\end{itemize}
\end{lemm}

\begin{proof} 
The result follows directly from the description of the actions given previously.
\end{proof}

As usually for $\sigma$ a representation of $W_v$ and $n \in \frac{1}{2} \Zm$, we will denote $\sigma(n)$
the twisted representation $g \mapsto \sigma(g) |\art^{-1}_v(g)|^n$ where $|-|$ is the absolute value of $F_v$.

\subsection{Free perverse sheaf}

Let $S=\spec \Fm_q$ and $X/S$ of finite type, then the usual $t$-structure on
$\DC(X,\overline \Zm_l):=D^b_c(X,\overline \Zm_l)$ is
$$\begin{array}{l}
A \in \lexp p \DC^{\leq 0}(X,\overline \Zm_l)
\Leftrightarrow \forall x \in X,~\hi^k i_x^* A=0,~\forall k >- \dim \overline{\{ x \} } \\
A \in \lexp p \DC^{\geq 0}(X,\overline \Zm_l) \Leftrightarrow \forall x \in X,~\
\hi^k i_x^! A=0,~\forall k <- \dim \overline{\{ x \} }
\end{array}$$
where $i_x:\spec \kappa(x) \hookrightarrow X$ and $\hi^k(K)$ is the $k$-th sheaf 
of cohomology of $K$.

\begin{nota} 
Let denote $\lexp p \CC(X,\overline \Zm_l)$ the heart of this $t$-structure with associated 
cohomology functors $\lexp p \hi^i$. For a functor $T$ we denote 
$\lexp p T:=\lexp p \hi^0 \circ T$.
\end{nota}

The category  $\lexp p \CC(X,\overline \Zm_l)$ is abelian equipped with a torsion theory 
$(\TC,\FC)$ where $\TC$ (resp. $\FC$) is the full subcategory of objects $T$ (resp. $F$) 
such that $l^N 1_T$ is trivial for some large enough $N$(resp. $l.1_F$ is a monomorphism). 
Applying Grothendieck-Verdier duality, we obtain
$$\begin{array}{l}
\lexp {p+} \DC^{\leq 0}(X,\overline \Zm_l):= \{ A \in \lexp p \DC^{\leq 1}(X,\overline \Zm_l):~
\lexp p \hi^1(A) \in \TC \} \\
\lexp {p+} \DC^{\geq 0}(X,\overline \Zm_l):= \{ A \in \lexp p \DC^{\geq 0}(X,\overline \Zm_l):~
\lexp p \hi^0(A) \in \FC \} \\
\end{array}$$
with heart$\lexp {p+} \CC(X,\overline \Zm_l)$  equipped with its torsion theory
$(\FC,\TC[-1])$.

\begin{defin} (cf. \cite{boyer-torsion} \S 1.3) \label{defi-FC}
Let
$$\FC(X,\overline \Zm_l):=\lexp p \CC(X,\overline \Zm_l) \cap \lexp {p+} \CC(X,\overline \Zm_l)
=\lexp p \DC^{\leq 0}(X,\overline \Zm_l) \cap \lexp {p+} \DC^{\geq 0}(X,\overline \Zm_l)$$
the quasi-abelian category of free perverse sheaves over $X$.
\end{defin}

\noindent \textit{Remark}: for an object $L$ of $\FC(X,\overline \Zm_l)$, 
we will consider filtrations
$$L_1 \subset L_2 \subset \cdots \subset L_e=L$$ 
such that for every $1 \leq i \leq e-1$, $L_i \hookrightarrow L_{i+1}$ is a
strict monomorphism, i.e. $L_{i+1}/L_i$ is an object of $\FC(X,\overline \Zm_l)$.

For a free $L \in \FC(X,\Lambda)$, we consider the following diagram
$$\xymatrix{ 
& L \ar[drr]^{\can_{*,L}} \\
\lexp {p+} j_! j^* L \ar[ur]^{\can_{!,L}} 
\ar@{->>}[r]|-{+} & \lexp p j_{!*}j^* L \ar@{^{(}->>}[r]_+ & 
\lexp {p+} j_{!*} j^* L \ar@{^{(}->}[r]|-{+} & \lexp p j_*j^* L
}$$
where below is, cf. the remark following 1.3.12 of \cite{boyer-torsion},
the canonical factorisation of $\lexp {p+} j_! j^* L \longrightarrow \lexp {p} j_* j^* L$
and where the maps $\can_{!,L}$ and $\can_{*,L}$ are given by the adjonction property.
Consider now $X$ equipped with a stratification 
$$X=X^{\geq 1} \supset X^{\geq 2} \supset \cdots \supset X^{\geq d,}$$
and let $L \in \FC(X,\overline \Zm_l)$.
For $1 \leq h <d$, let denote
$X^{1 \leq h}:=X^{\geq 1}-X^{\geq h+1}$ and $j^{1 \leq h}:X^{1†\leq h} \hookrightarrow
X^{\geq 1}$. We then define
$$\Fil^r_{!}(L):=\im_\FC \Bigl ( \lexp {p+} j^{1 \leq r}_! j^{1 \leq r,*}L \longrightarrow L\Bigr ),$$
which gives a filtration
$$0=\Fil^{0}_{!}(L) \subset \Fil^1_{!}(L) \subset \Fil^1_{!}(L) \cdots \subset \Fil^{d-1}_{!}(L) 
\subset \Fil^d_{!}(L)=L.$$
Dually
$\CoFil_{*,r}(L)=\coim_\FC\Bigl ( L \longrightarrow \lexp {p} j^{1 \leq r}_* j^{1 \leq r,*}L \Bigr ),$
define a cofiltration
\begin{multline*}
L = \CoFil_{\SF,*,d}(L) \twoheadrightarrow \CoFil_{\SF,*,d-1}(L) \twoheadrightarrow \cdots \\
\cdots
\twoheadrightarrow \CoFil_{\SF,*,1}(L) \twoheadrightarrow \CoFil_{\SF,*,0}(L)=0,
\end{multline*}
and a filtration
$$0=\Fil_{*}^{-d}(L) \subset \Fil_{*}^{1-d}(L) \subset \cdots \subset \Fil_{*}^0(L)=L$$
where $\Fil_{*}^{-r}(L):=\ker_\FC \bigl ( L \twoheadrightarrow \CoFil_{*,r}(L) \bigr ).$

\rem These two constructions are exchanged by Grothendieck-Verdier duality, 
$D \Bigl ( \CoFil_{\SF,!,-r}(L) \Bigr ) \simeq \Fil^{-r}_{\SF,*}( D(L))$ and
$D \Bigl ( \CoFil_{\SF,*,r}(L) \Bigr ) \simeq \Fil^{r}_{\SF,!}( D(L)).$

We can also refine the previous filtrations,
cf. \cite{boyer-torsion} proposition 2.3.3, to obtain exhaustive filtrations
\addtocounter{thm}{1}
\begin{multline} \label{eq-fill}
0=\Fill^{-2^{d-1}}_{!}(L) \subset \Fill^{-2^{d-1}+1}_{!}(L) \subset \cdots  \\
\cdots \subset
\Fill^0_{!}(L) \subset \cdots \subset \Fill^{2^{d-1}-1}_{!}(L)=L,
\end{multline}
such that the graduate $\grr^k(L)$ are simple over $\overline \Qm_l$, 
i.e. verify $\lexp p j^{= h}_{!*} j^{= h,*} \grr^k(L)  \htarrow_+ \grr^k(L)$
for some $h$.
Dually we can construct a cofiltration
$$L=\CoFill_{*,2^{d-1}}(L) \twoheadrightarrow \CoFill_{*,2^{d-1}-1}(L)
\twoheadrightarrow \cdots \twoheadrightarrow \CoFill_{*,-2^{d-1}}(L)=0$$
and a filtration $\Fill_{*}^{-r}(L):=\ker_\FC \bigl ( L \twoheadrightarrow \CoFill_{*,r}(L) \bigr )$.

\subsection{Vanishing cycles perverse sheaf}
\label{para-entier}

\begin{nota}
For $I \in \IC$, let 
$$\Psi_{I,\Lambda}:=R\Psi_{\eta_v,I}(\Lambda[d-1])(\frac{d-1}{2})$$
be the vanishing cycle autodual perverse sheaf on $X_{I,\bar s_v}$.
When $\Lambda=\overline \Zm_l$, we will simply write $\Psi_\IC$.
\end{nota}

Recall the following result of \cite{h-t} relating $\Psi_{\IC}$ with Harris-Taylor local systems.

\begin{propo} \label{prop-fbartau} 
(cf. \cite{h-t} proposition IV.2.2 and \S 2.4 of \cite{boyer-invent2}) \\
There is an isomorphism 
$G(\Am^{\oo,v}) \times P_{h,d-h}(F_v) \times W_v$-equivariant
$$\ind_{(D_{v,h}^\times)^0 \varpi_v^\Zm}^{D_{v,h}^\times} 
\bigl ( \hi^{h-d-i} \Psi_{\IC,\overline \Zm_l} \bigr )_{|X^{=h}_{\IC,\bar s,\overline{1_h}}} \simeq
\bigoplus_{\bar \tau \in \RC_{ \overline \Fm_l}(h)} \LC_{\overline \Zm_l,\overline{1_h}}
(\UC^{h-1-i}_{\bar \tau,\Nm}),$$
where
\begin{itemize}
\item $\RC_{ \overline \Fm_l}(h)$ is the set of equivalence classes of irreducible
$\overline \Fm_l$-representations of $D_{v,h}^\times$;

\item for $\bar \tau \in \RC_{ \overline \Fm_l}(h)$ and $V$ a 
$\overline \Zm_l$-representation of $D_{v,h}^\times$, then $V_{\bar \tau}$
denotes, cf. \cite{dat-torsion} \S B.2, the direct factor of $V$ whose irreducible
subquotients are isomorphic to a subquotient of $\bar \tau_{|\DC^\times_{v,h}}$
where $\DC_{v,h}$ is the maximal order of $D_{v,h}$.

\item With the previous notation, $\UC^i_{\bar \tau,\Nm}:= 
\big ( \UC^i_{F_v,\overline \Zm_l,d} \big )_{\bar \tau}$.

\item The matching between the system indexed by $\IC$ and those by $\Nm$
is given by the application $m_1:\IC \longrightarrow \Nm$.
\end{itemize}
\end{propo}

\rem For $\bar \tau \in \RC_{\overline \Fm_l}(h)$, and a lifting $\tau$
which by Jacquet-Langlands correspondence can be written $\tau \simeq \pi[t]_D$
for $\pi$ irreductible cuspidal, let $\varrho \in \scusp_{\overline \Fm_l}(g)$
be in the supercuspidal support. Then the inertial class of $\varrho$ depends only on
$\bar \tau$ and we will use the following notation.

\begin{nota}
With the previous notation, we denote $V_\varrho$ for $V_{\bar \tau}$. 
\end{nota}

\rem 
$\Psi_{\IC, \overline \Zm_l}$ is an object of $\FC(X_{\IC,\bar s}, \overline \Zm_l)$.
Indeed, by \cite{ast} proposition 4.4.2,
$\Psi_{\IC, \overline \Zm_l}$ is an object of 
$\lexp p \DC^{\leq 0}(X_{\IC,\bar s},\overline \Zm_l)$. By
\cite{ill} variant 4.4 of theorem 4.2, we have 
$D \Psi_{\IC, \overline \Zm_l} \simeq \Psi_{\IC, \overline \Zm_l}$, so that
$$\Psi_{\IC, \overline \Zm_l} \in \lexp p \DC^{\leq 0}(X_{\IC,\bar s},\overline \Zm_l) \cap 
\lexp {p+} \DC^{\geq 0}(X_{\IC,\bar s},\overline \Zm_l)=\FC(X_{\IC,\bar s}, \overline \Zm_l).$$

\begin{propo} \label{prop-decomposition} (cf. \cite{boyer-local-ihara} \S 3.2)
We have a decomposition
$$\Psi_\IC \simeq \bigoplus_{g=1}^d \bigoplus_{\varrho \in \scusp_{\overline \Fm_l}(g)} \Psi_{\varrho}$$
where all the Harris-Taylor perverse sheaves of $\Psi_\varrho \otimes_{\overline \Zm_l} \overline \Qm_l$
are of type $\varrho$.
\end{propo}

\rem In \cite{boyer-invent2}, we decomposed $\Psi \otimes_{\overline \Zm_l} \overline \Qm_l$
as a direct sum $\bigoplus_{\pi_v} \Psi_{\pi_v}$ where $\pi_v$ described the set of equivalent
inertial classes of irreducible cuspidal representations. The
$\Psi_\varrho \otimes_{\overline \Zm_l} \overline \Qm_l \simeq 
\bigoplus_{\pi_v \in \cusp(\varrho)} \Psi_{\pi_v}$ where $\cusp(\varrho)$ is the set of equivalent inertial
classes of irreducible cuspidal representation of type $\varrho$ in the sense of definition \ref{defi-rho-type}.

In \cite{boyer-FT}, we give the precise description of the $\gr^r_{\SF,!}(\Psi_{\IC,\varrho})$
which is defined other $\overline \Zm_l$.
By construction they are supported on $X^{\geq r}_{\IC,\bar s_v}$ and trivial if $g$ 
does not divide $r$. Otherwise for $r=tg$, we have
\begin{multline*}
\ind_{(D_{v,tg}^\times)^0 \varpi_v^\Zm}^{D_{v,tg}^\times} \Bigl (  
j^{=tg,*} \gr^{tg}_{\SF,!}(\Psi_{\IC,\varrho}) \otimes_{\overline \Zm_l} \overline \Qm_l \Bigr )
\simeq \\ \bigoplus_{\atop{i=-1}{t_ig_i(\varrho)= tg}}^{r(\varrho)} 
\bigoplus_{\pi_v \in \scusp_i(\varrho)} HT(\pi_v,\st_{t_i}(\pi_v)) \otimes \Lm(\pi_v)(-\frac{t_i-1}{2}).
\end{multline*}
We can then consider the following naive $\varrho$-filtration
$$\Fil^*_{\varrho,r(\varrho),\overline \Qm_l} (\Psi,tg) \subset \cdots \subset 
\Fil^*_{\varrho,-1,\overline \Qm_l} (\Psi,tg) =j^{=tg,*} \gr^{tg}_{\SF,!}(\Psi_{\IC,\varrho}) 
\otimes_{\overline \Zm_l}  \overline \Qm_l$$
where $\ind_{(D_{v,tg}^\times)^0 \varpi_v^\Zm}^{D_{v,tg}^\times} \Bigl (  
\Fil^*_{\varrho,k,\overline \Qm_l} (\Psi,tg) \Bigr )$ is isomorphic to
$$\bigoplus_{\atop{i=k}{t_ig_i(\varrho) = tg}}^{r(\varrho)} 
\bigoplus_{\pi_v \in \scusp_i(\varrho)} HT(\pi_v,\st_{t_i}(\pi_v)) \otimes \Lm(\pi_v)(-\frac{t_i-1}{2}),$$
and the associated entire filtration of $j^{=tg,*} \gr^{tg}_{\SF,!}(\Psi_{\IC,\varrho})$ defined by pullback
$$\xymatrix{
\Fil^*_{\varrho,k} (\Psi,tg) \ar@{^{(}-->}[r] \ar@{^{(}-->}[r] \ar@{^{(}-->}[d] 
& \Fil^*_{\varrho,k,\overline \Qm_l} (\Psi,tg) \ar@{^{(}->}[d] \\
j^{=tg,*} \gr^{tg}_{\SF,!}(\Psi_{\IC,\varrho}) \ar@{^{(}->}[r] & j^{=tg,*} \gr^{tg}_{\SF,!}(\Psi_{\IC,\varrho})
\otimes_{\overline \Zm_l} \overline \Qm_l.
}$$
For $k=-1,\cdots,r(\varrho)$, the graduates $\gr_{\varrho,k} (\Psi,tg)$ are then of $\varrho$-type $k$.
We can then refine these filtrations by separating the $\pi_v \in \scusp_k(\varrho)$ to obtain
$$(0)=\Fil_{\varrho}^{*,0}(\Psi,tg) \subset \Fil_{\varrho}^{*,1}(\Psi,tg)  \subset \cdots \subset 
\Fil_{\varrho}^{*,r}(\Psi,tg)=j^{=tg,*} \gr^{tg}_{\SF,!}(\Psi_{\IC,\varrho}).$$
By taking the iterated images of $j^{=tg}_!  \Fil_{\varrho}^k(\Psi,tg) \longrightarrow 
\gr^{tg}_{\SF,!}(\Psi_{\IC,\varrho})$, we then construct a filtration
$$(0)=\Fil_{\varrho}^{0}(\Psi,tg) \subset \Fil_{\varrho}^{1}(\Psi,tg)  \subset \cdots \subset 
\Fil_{\varrho}^{r}(\Psi,tg)= \gr^{tg}_{\SF,!}(\Psi_{\IC,\varrho}).$$
Finally we can filtrate each of these graduate using an exhaustive filtration of stratification to obtain
a filtration of $\Psi_{\IC,\varrho}$ whose graduates are the $\PF(\pi_v,t)(\frac{1-s_i(\varrho)+2k}{2})$ for 
$\pi_v \in \scusp_i(\varrho)$ with $i \geq -1$, and $k=0,\cdots,s_i(\varrho)-1$.
%
%
%

\section{Cohomology of KHT Shimura varieties}

\subsection{Localization at a non pseudo-Eisenstein ideal}

\begin{defin} \label{defi-spl}
Define $\spl$ the set of  places $v$ of $F$ such that $p_v:=v_{|\Qm} \neq l$ is split in $E$ and
$B_v^\times \simeq GL_d(F_v)$.  For each $I \in \IC$, write
$\spl(I)$ the subset of $\spl$ of places which doesn't divide the level $I$.
\end{defin}

Let $\unr(I)$ be the union of
\begin{itemize}
\item places $q \neq l$ of $\Qm$ inert in $E$ not below a place of $\bad$ and where $I_q$ is maximal,

\item and places $w \in \spl(I)$.
\end{itemize}

\begin{nota} \label{nota-spl2}
For $I \in \IC$ a finite level, write
$$\Tm_I:=\prod_{x \in \unr(I)} \Tm_{x}$$
where for $x$ a place of $\Qm$ (resp. $x \in \spl(I)$), 
$\Tm_{x}$ is the unramified Hecke algebra of $G(\Qm_x)$ (resp. of $GL_d(F_x)$) over
$\overline \Zm_l$.
\end{nota}

\noindent \textit{Example}:  for $w \in \spl(I)$, we have
$$\Tm_{w}=\overline \Zm_l \bigl [T_{w,i}:~ i=1,\cdots,d \bigr ],$$
where $T_{w,i}$ is the characteristic function of
$$GL_d(\OC_w) \diag(\overbrace{\varpi_w,\cdots,\varpi_w}^{i}, \overbrace{1,\cdots,1}^{d-i} ) 
GL_d(\OC_w) \subset  GL_d(F_w).$$
More generally, the Satake isomorphism identifies $\Tm_x$ with $\overline \Zm_l[X^{un}(T_x)]^{W_x}$
where
\begin{itemize}
\item $T_x$ is a split torus,

\item $W_x$ is the spherical Weyl group

\item and $X^{un}(T_x)$ is the set of $\overline \Zm_l$-unramified characters of $T_x$.
\end{itemize}

Consider a fixed maximal ideal $\mathfrak m$ of $\Tm_I$ and for every $x \in \unr(I)$ let denote
$S_{\mathfrak m}(x)$ be the multi-set\footnote{A multi-set is a set with multiplicities.} 
of modulo $l$ Satake parameters at $x$ associated to $\mathfrak m$.

\noindent \textit{Example}: 
for every $w \in \spl(I)$, the multi-set of Satake parameters at $w$ corresponds to the roots of
the Hecke polynomial
$$P_{\mathfrak{m},w}(X):=\sum_{i=0}^d(-1)^i q_w^{\frac{i(i-1)}{2}} \overline{T_{w,i}} X^{d-i} \in \overline 
\Fm_l[X]$$
i.e.
$S_{\mathfrak{m}}(w) := \bigl \{ \lambda \in \Tm_I/\mathfrak m \simeq \overline \Fm_l \hbox{ such that }
P_{\mathfrak{m},w}(\lambda)=0 \bigr \}.$
For a maximal ideal $\widetilde{\mathfrak m}$ of $\Tm_{I^l} \otimes_{\overline \Zm_l} \overline \Qm_l$, we also
have the multi-set of Satake parameters 
$$S_{\widetilde{\mathfrak{m}}}(w) := \bigl \{ \lambda \in 
\Tm_I \otimes_{\overline \Zm_l} \overline \Qm_l/ \widetilde{\mathfrak m} \simeq \overline \Qm_l \hbox{ such that }
P_{\widetilde{\mathfrak{m},w}}(\lambda)=0 \bigr \}.$$

\begin{nota}
Let $\Pi$ be an irreducible automorphic representation of $G(\Am)$ of level $I$ which means here, 
that $\Pi$ has non trivial invariants under $I$ and
 for every $x \in \unr(I)$, then $\Pi_x$ is unramified. Then $\Pi$ defines 
 \begin{itemize}
 \item a maximal ideal 
$\widetilde{\mathfrak m}(\Pi)$ of $\Tm_{I^l} \otimes_{\overline \Zm_l} \overline \Qm_l$, or

\item a minimal prime ideal $\widetilde{\mathfrak m}(\Pi)$ of $\Tm_{I^l}$ contained in a maximal ideal
$\mathfrak m(\Pi)$ of $\Tm_{I^l}$ which corresponds to its the modulo $l$ Satake parameters.
\end{itemize}
\end{nota}

A minimal prime ideal $\widetilde{\mathfrak m}$ of $\Tm_{I^l}$ is said cohomological if there exists
a cohomological automorphic $\overline \Qm_l$-representation $\Pi$ of $G(\Am)$ of level $I$
with $\widetilde{\mathfrak m}=\widetilde{\mathfrak m}(\Pi)$. Such a $\Pi$ is not unique but
$\widetilde{\mathfrak m}$ defines a unique near equivalence class in the sense of \cite{y-t}, we denote it
$\Pi_{\widetilde{\mathfrak m}}$. Let then
$$\rho_{\widetilde{\mathfrak m},\overline \Qm_l}:\gal(\bar F/F) \longrightarrow GL_d(\overline \Qm_l)$$ 
be the galoisian representation associated to such a $\Pi$ thanks to \cite{h-t} and \cite{y-t}, 
which by Cebotarev theorem, can be defined over some finite extension $K_{\widetilde{\mathfrak m}}$, i.e.
$\rho_{\widetilde{\mathfrak m},\overline \Qm_l} \simeq \rho_{\widetilde{\mathfrak m}} 
\otimes_{K_{\widetilde{\mathfrak m}}} \overline \Qm_l$.

It's well known that  $\rho_{\widetilde{\mathfrak m}}$ has stable lattices and the semi-simplication
of its modulo $l$ reduction is independent of the chosen stable lattice. Moreover it depends only
of the maximal ideal $\mathfrak m$, we denote
$$\overline \rho_{\mathfrak m}: G_F \longrightarrow GL_d(\overline \Fm_l),$$
its extension to $\overline \Fm_l$. 
For every $w \in \spl(I)$, recall that the multiset of eigenvalues of $\overline \rho_{\mathfrak m}(\frob_w)$
is $S_{\mathfrak m}(w)$ obtained from $S_{\widetilde{\mathfrak m}}(w)$ by taking modulo $l$ reduction.

Assume moreover that $\overline \rho_{\mathfrak m}$ is absolutely 
irreducible. Then the $\overline \Qm_l$-cohomology group 
$H^{d-1}(X_{U,\bar \eta},\overline \Qm_l)_{\mathfrak m}$ gives a 
continuous $d$-dimensional Galois representation
$$\rho_{\mathfrak m}:\gal_{F,S} \longrightarrow GL_d(\Tm_{S,\mathfrak m}[1/l]),$$
where $\gal_{F,S}$ is the Galois group of the maximal extension of $F$ which
is unramified outside $S$. As all characteristic polynomials of Frobenius takes values
in $\Tm_{S,\mathfrak m}$ and as $\overline \rho_{\mathfrak m}$ is absolutely irreducible,
using classical theory of pseudo-representations, we know that $\rho_{\mathfrak m}$
takes values in $GL_d(\Tm_{S,\mathfrak m})$.

\subsection{Freeness of the cohomology}

From now on we fix such a maximal ideal $\mathfrak m$ of $\Tm_I$ verifying one of 
the following three conditions:
\begin{itemize}
\item[(1)] There exists $w_1 \in \spl(I)$ such that $S_{\mathfrak{m}}(w_1)$ does not contain any 
sub-multi-set of the shape $\{ \alpha, q_{w_1} \alpha \}$ where $q_{w_1}$ is the cardinal
of the residue field. Note also that this hypothesis is closed but weaker than
the notion of \emph{decomposed genericness} introduced in \cite{scholze-cara}.

\item[(2)] With $l \geq d+2$ we ask $\mathfrak m$ to verify one of the two following hypothesis
\begin{itemize}
\item either $\overline \rho_{\mathfrak m}$ is induced form a character of $G_K$ for
a cyclic galoisian extension $K/F$;

\item either $SL_n(k) \subset \overline \rho_{\mathfrak m}(G_F) \subset \overline \Fm_l^\times
GL_n(k)$ for a sub-field $k \subset \overline \Fm_l$.
\end{itemize}

\item[(3)] $\overline \rho_{\mathfrak m}$ is absolutely irreducible and $l \geq d+1$
except maybe for a finite number of them depending only of the extension $F/\Qm$, cf.
the main theorem of \cite{boyer-bordeaux}.
\end{itemize}

\begin{theo} \label{theo2} (cf. \cite{boyer-imj})
For $\mathfrak m$ as above, the localized cohomology groups 
$H^i(X_{I,\bar \eta},\overline \Zm_l)_{\mathfrak m}$ are free. 
\end{theo}

As $X_I \longrightarrow \spec\OC_v$ is proper, we have a $G(\Am^\oo) \times W_v$-equivariant
isomorphism
$H^{d-1+i}(X_{\IC,\bar \eta_v},\overline \Zm_l) \simeq H^i(X_{I,\bar s_v},\Psi_{\IC}).$
Using the previous filtration of $\Psi_{\IC}$, we can compute 
$H^{p+q}(X_{I,\bar s_v},\Psi_{\IC,\varrho})_{\mathfrak m}$
through a spectral sequence whose entries $E_1^{p,q}$ are the 
$H^j(X_{\IC,\bar s_v},\PF(\pi_v,t)(\frac{1-s_i(\varrho)+2k}{2}))_{\mathfrak m}$
for $\pi_v \in \scusp_i(\varrho)$ with $i \geq -1$, and $k=0,\cdots,s_i(\varrho)-1$.
Over $\overline \Qm_l$, it follows from \cite{boyer-compositio}, thanks to the hypothesis (1) above on 
$\mathfrak m$, that all these cohomology groups are concentrated in degree $0$ so that this 
$\overline \Qm_l$-spectral sequence degenerates in $E_1$. In this section, under (H2), 
we want to prove the same result on $\overline \Fm_l$ which is equivalent to the freeness of the 
$H^j(X_{\IC,\bar s_v},\PF(\pi_v,t)(\frac{1-s_i(\varrho)+2k}{2}))_{\mathfrak m}$.

We need first some notations from \cite{boyer-compositio} \S 1.2. For all $t \geq 0$, we denote
$$\Gamma^t:=\Bigl \{ (a_1,\cdots,a_r,\epsilon_1,\cdots,\epsilon_r) \in \Nm^r \times \{ \pm \}^r:~
r \geq 1,~\sum_{i=1}^r a_i=t \Bigr \}.$$
A element of $\Gamma^t$ will be denoted $(\overleftarrow{a_1}, \cdots , \overrightarrow{a_r})$ where
for the arrow above each integer $a_i$ is $\overleftarrow{a_i}$ (resp.
$\overrightarrow{a_i}$) if $\epsilon_i=-$ (resp. $\epsilon_i=+$). We then consider on $\Gamma^t$ 
the equivalence relation induced by  
$$(\cdots,\overleftarrow{a},\overleftarrow{b},\cdots)=(\cdots,\overleftarrow{a+b},\cdots), \quad
(\cdots,\overrightarrow{a},\overrightarrow{b},\cdots)=(\cdots,\overrightarrow{a+b},\cdots),$$ 
and $(\cdots, \overleftarrow{0},\cdots)=(\cdots,\overrightarrow{0},\cdots)$. 
We denote $\overrightarrow{\Gamma}^t$ the set of these equivalence classes whose elements are
denoted $[\overleftarrow{a_1},\cdots,\overrightarrow{a_r}]$.

\rem In each class $[\overleftarrow{a_1},\cdots,\overrightarrow{a_k}] 
\in \overrightarrow{\Gamma}^t$, there exists a unique reduced element 
$(b_1,\cdots,b_r,\epsilon_1,\cdots,\epsilon_r) \in \Gamma^t$ such that  for all $1 \leq i \leq r$
$b_i>0$ and for $1 \leq i < r$, $\epsilon_i \epsilon_{i+1}=-$.

\begin{defi} \label{defi-graphe}
Let $(b_1,\cdots,b_r,\epsilon_1,\cdots,\epsilon_r)$ be the reduced element in 
$[\overleftarrow{a_1},\cdots,\overrightarrow{a_k}] \in \overrightarrow{\Gamma}^t$. We then define 
$$\SC\Bigl ( [\overleftarrow{a_1},\cdots,\overrightarrow{a_k}] \Bigr )$$
as the subset of permutations $\sigma$ of $\{ 0,\cdots,t-1 \}$ such that for all 
$1 \leq i \leq r$ with $\epsilon_i=+$ (resp. $\epsilon_i=-$) and for all
 $b_1+\cdots + b_{i-1} \leq k < k' \leq b_1+ \cdots +b_i$ then $\sigma^{-1}(k) < \sigma^{-1}(k')$ 
(resp. $\sigma^{-1}(k) > \sigma^{-1}(k')$). 

We also introduce $\SC^{op}\Bigl ( [\overleftarrow{a_1},\cdots,\overrightarrow{a_k}] \Bigr )$, by imposing
under the same conditions,  $\sigma^{-1}(k)> \sigma^{-1}(k')$ (resp. $\sigma^{-1}(k) < \sigma^{-1}(k')$).
\end{defi}

\begin{prop} \label{prop-defi-gamma-jacquet} (cf. \cite{zelevinski2} \S 2)
Let $g$ be a divisor of $d=sg$ and $\pi$ an irreducible cuspidal representation of $GL_g(F_v)$. 
There exists a bijection
$$[\overleftarrow{a_1}, \cdots , \overrightarrow{a_r}] \in \overrightarrow{\Gamma}^{s-1}
\mapsto [\overleftarrow{a_1}, \cdots , \overrightarrow{a_r}]_{\pi}$$ 
into the set of irreducible sub-quotients of the induced representation
$$\pi \{ \frac{1-s}{2} \} \times \pi \{ \frac{3-s}{2} \} \times \cdots \times \pi \{ \frac{s-1}{2} \}$$
characterized by the following property
$$J_{P_{g,2g,\cdots,sg}}([\overleftarrow{a_1}, \cdots , \overrightarrow{a_r}]_{\pi})=\sum_{\sigma \in
\SC \Bigl ( [\overleftarrow{a_1}, \cdots , \overrightarrow{a_r}] \Bigr ) } \pi \{ \frac{1-s}{2}+\sigma(0) \}
\otimes \cdots \otimes \pi\{ \frac{1-s}{2}+\sigma(s-1) \},$$
or equivalently by
$$J_{P^{op}_{g,2g,\cdots,sg}}([\overleftarrow{a_1}, \cdots , \overrightarrow{a_r}]_{\pi})=\sum_{\sigma \in
\SC^{op} \Bigl ( [\overleftarrow{a_1}, \cdots , \overrightarrow{a_r}] \Bigr ) } \pi \{ \frac{1-s}{2}+\sigma(0) \}
\otimes \cdots \otimes \pi\{ \frac{1-s}{2}+\sigma(s-1) \}.$$
\end{prop}

\rem With these notations $\st_s(\pi)$ (resp. $\speh_s(\pi)$) is $[\overleftarrow{s-1}]_\pi$ 
(resp. $[\overrightarrow{s-1}]_\pi$).

\begin{lemm} \label{lem-redmodl}
Let $\pi$ be an irreducible cuspidal $\overline \Qm_l$-representation 
of $GL_g(F_v)$ such that its modulo $l$ reduction $\varrho$ is supercuspidal.
Suppose the cardinal of the Zelevinsky line of $\varrho$ is greater or equal to $s$.
Then the irreducible sub-quotients of the modulo $l$ reduction of   
$[\overleftarrow{a_1}, \cdots , \overrightarrow{a_r}]_{\pi}$ for
$[\overleftarrow{a_1}, \cdots , \overrightarrow{a_r}]$ describing $\overrightarrow{\Gamma}^{s-1}$,
are pairwise distinct.
%
\end{lemm}

\begin{proof}
By hypothesis on the cardinal of the Zelevinsky line,
all these irreducible $\overline \Fm_l$-sub-quotient have a non trivial image under
$J_{P_{g,2g,\cdots,sg}}$. The result then follows directly
\begin{itemize}
\item from the commutation of Jacquet functors with the modulo $l$ reduction and 

\item from the fact that the $r_l(\pi) \{ \frac{1-s}{2}+k \}$ are pairwise distinct for $0 \leq k < s$
so that the image under $J_{P_{g,2g,\cdots,sg}}$ of
$\pi \{ \frac{1-s}{2} \} \times \pi \{ \frac{3-s}{2} \} \times \cdots \times \pi \{ \frac{s-1}{2} \}$
is multiplicity free.
\end{itemize}
\end{proof}

\begin{nota} \label{nota-segment-modl}
We will denote $[\overleftarrow{a_1}, \cdots , \overrightarrow{a_r}]_{\varrho}$ any irreducible sub-quotient
of the modulo $l$ reduction of $[\overleftarrow{a_1}, \cdots , \overrightarrow{a_r}]_{\pi}$.
\end{nota}

\rem If moreover the cardinal of the Zelevinsky line of $\varrho$ is strictly greater than $s$, then, cf.
\cite{boyer-repmodl}, the modulo $l$ reduction of $[\overleftarrow{s-1}]_\pi$ is irreducible and non
degenerate, i.e. $[\overleftarrow{s-1}]_\varrho$ is well determinate and non degenerate.

\begin{theo} \label{theo-torsion}
Consider a maximal ideal $\mathfrak m$ of $\Tm_S$ such that for all $i$, the
$\overline \Zm_l$-module $H^i(X_{U,\bar \eta},\overline \Zm_l)_{\mathfrak m}$ is free.
We suppose moreover (H2) that the image of $\bar \rho_{\mathfrak m,w_0}$ in the
Grothendieck group is multiplicity free. Then for all $\overline \Zm_l$-Harris-Taylor local system 
$HT(\pi_{w_0},t)$, the $H^i(X_{U,\bar s_{w_0}},\lexp pj^{=tg}_{!*} HT(\pi_{w_0},t))_{\mathfrak m}$ 
are free.
\end{theo}

\rem In the previous statement, we just need the multiplicity free part of (H2) as it is used in 
the previous lemma. 
Note moreover that, by \cite{boyer-aif} \S 4.5, the multiplicity free hypothesis is necessary.

\begin{proof}
First denote $\scusp_{w_0}(\mathfrak m)$ the set of inertial equivalence classes of 
irreducible supercuspidal
$\overline \Fm_l$-representations belonging to the supercuspidal support of the modulo $l$
reduction of the local component at $w_0$ of a representation $\Pi$ in the near equivalence
class $\Pi_{\mathfrak m}$ associated to $\mathfrak m$. 

We then consider the vanishing cycle spectral sequence at $w_0$, localized at $\mathfrak m$
$$H^i(X_{U,\bar \eta_{w_0}},\overline \Zm_l)_{\mathfrak m} \simeq 
\bigoplus_{\varrho \in \scusp_{w_0}(\mathfrak m)}  H^i(X_{U,\bar s_{w_0}},
\Psi_{\IC,\varrho})_{\mathfrak m}.$$
Then for every $\varrho \in \scusp_{w_0}(\mathfrak m)$, the $H^i(X_{U,\bar s_{w_0}},
\Psi_{\IC,\varrho})_{\mathfrak m}$ are free. For $\pi_v$ of type $\varrho$, the strategy to
prove the freeness of the  $H^i(X_{U,\bar s_{w_0}},\lexp pj^{=tg}_{!*} HT(\pi_{w_0},t))_{\mathfrak m}$
is then to argue by absurdity and to produce some torsion cohomology class in one of the
 $H^i(X_{U,\bar s_{w_0}}, \Psi_{\IC,\varrho})_{\mathfrak m}$.
Let then $t$ be minimal such that there exists $i \neq 0$ with 
$$H^{i}(X_{\IC,\bar s_{w_0}},\PC(\pi_{w_0},t)(\frac{1-t+2k}{2}))_{\mathfrak m} 
\otimes_{\overline \Zm_l} \overline \Fm_l \neq (0)$$
for $0 \leq k < t$, and where $\PC(\pi_{w_0},t)(\frac{1-t+2k}{2})$ is a graduate of some filtration
of $\Psi_\varrho$.

Consider for example the filtration constructed before using the adjunction 
$j^{1 \leq h}_! j^{1 \leq h,*} \longrightarrow \Id$. As remarked before, even at taking $\Psi_{\varrho^\vee}$
and its filtration constructed using $\Id \longrightarrow j^{1 \leq h}_* j^{1 \leq h,*}$, we can suppose that such
$i$ is strictly negative and we denote $i_0$ such a minimal $i$.

By lemma \ref{lem-coho-flou}, an irreducible $GL_d(F_{w_0}) \times W_{w_0}$-equivariant sub-quotient
of $H^{i}(X_{\IC,\bar s_{w_0}},\PC(\pi_{w_0},t)(\frac{1-t+2k}{2}))_{\mathfrak m} 
\otimes_{\overline \Zm_l} \overline \Fm_l$ is one of the modulo $l$ reduction of a representation
we can write in the following shape
$$\Bigl ( [\overleftarrow{a_1}, \cdots , \overrightarrow{a_{i-1}}, \overleftrightarrow{1},
\overleftarrow{t-1}, \overleftrightarrow{1}, \overrightarrow{a_{i+1}},
\cdots \overleftarrow{a_r}]_{\pi} 
\times \Upsilon_{w_0} \Bigr ) \otimes \Lm(\pi \{ \frac{\delta}{2} \} )$$
where the $a_j$ are some integers, $\Upsilon_{w_0}$ is an irreducible entire $\overline \Qm_l$-representation 
whose modulo $l$ reduction have a supercuspidal support away from those of the previous segment and where
\begin{itemize}
\item the symbol $\overleftrightarrow{1}$ before (resp. after) the $\overleftarrow{t-1}$ can be
$\overleftarrow{1}$ or $\overrightarrow{1}$ if $\sum_{j=1}^{i-1} a_i>0$ (resp.
$\sum_{j=i+1}^r a_i>0$). We will denote moreover $a_i=t+1$.


\item Let denote $\Bigl \{ \pi \{ \frac{\alpha}{2} \}, \pi \{ \frac{\alpha}{2} +1 \}, \cdots,
\pi \{ \frac{\alpha}{2}+t-1 \} \Bigr \}$ the supercuspidal support of 
$\overleftarrow{t-1}$ inside $[\overleftarrow{a_1}, \cdots , \overleftarrow{t-1},\cdots \overrightarrow{a_r}]_{\pi}$.
The $\frac{\delta}{2} =\frac{\alpha}{2}+k$ where $k$ is the integer in $\PC(\pi_{w_0},t)(\frac{1-t+2k}{2})$.
\end{itemize}

\rem In particular we can suppose that the previous $k$ is equal to $0$.

Consider then such an irreducible sub-quotient $\tau \times \psi_{w_0} \otimes  \sigma$ where
\begin{itemize}
\item $\psi_{w_0}$ (resp. $\sigma$) is any irreducible sub-quotient of the modulo $l$ reduction of  $\Upsilon_{w_0}$
(resp. $\Lm(\pi \{ \frac{\delta}{2} \} )$), and

\item $\tau$ is an irreducible sub-quotient of the modulo $l$ reduction of some
$$[\overleftarrow{a_1}, \cdots , \overrightarrow{a_{i-1}}, \overrightarrow{1},
\overleftarrow{t-1}, \overrightarrow{1}, \overrightarrow{a_{i+1}},
\cdots \overrightarrow{a_r}]_{\pi}.$$
By the previous lemma, we can recover the $a_i$ from $\tau$.
\end{itemize}
Let show now that this $\tau \times \psi_{w_0} \otimes  \sigma$ is also a sub-quotient of
$H^{i_0}(X_{\IC,\bar s_{w_0}},\Psi_{\IC,\varrho})_{\mathfrak m} \otimes_{\overline \Zm_l} \overline \Fm_l$, 
which contradicts our hypothesis on $\mathfrak m$.
Let denore $\Fil^{k-1} \subset \Fil^{k} \subset \Psi_{\IC,\varrho}$ such that
$\gr^k=\Fil^k/\Fil^{k-1} \simeq \PC(\pi_{w_0},t)(\frac{1-t}{2})$.
By the hypothesis (H2), all the irreducible cuspidal $\overline \Qm_l$-representations 
$\pi'_{w_0} \in \scusp_{\overline \Fm_l}(\varrho)$
such that one of the $H^{i}(X_{\IC,\bar s_{w_0}},\PC(\pi'_{w_0},t)(\frac{1-t+2k}{2}))_{\mathfrak m} \neq (0)$
is necessarily of $\varrho$-type $-1$. Then in particular all the Harris-Taylor perverse sheaves $\PC(\pi'_{w_0},t')$
which are sub-quotient of $\Fil^{k-1}$, must verify $t'>t$. The spectral sequence which computes
$H^{i_0+1}(X_{\IC,\bar s_{w_0}},\Fil^{k-1})_{\mathfrak m} \otimes_{\overline \Zm_l} \overline \Fm_l$ 
thanks to a filtration of $\Fil^{k-1}$, allows to describe it as extensions between irreducible sub-quotients
of the modulo $l$ reduction of some
$$\Bigl ( [\overleftarrow{a_1}, \cdots , \overrightarrow{a_{i-1}}, \overleftrightarrow{1},
\overleftarrow{t'-1}, \overleftrightarrow{1}, \overrightarrow{a_{i+1}},
\cdots \overrightarrow{a_r}]_{\pi'} 
\times \psi_{w_0} \Bigr ) \otimes \Lm(\pi' \{ \frac{\delta'}{2} \} )$$
with $t'>t$ and where $\pi' \{\frac{\delta'}{2} \}$ belongs to the supercuspidal support of 
$\overleftarrow{t'-1}$ in the previous writing. But using the inequality $t'>t$, we see that
$\tau$ can't be a sub-quotient of the modulo $l$ reduction of any
$[\overleftarrow{a_1}, \cdots , \overrightarrow{a_{i-1}}, \overrightarrow{1},
\overleftarrow{t-1}, \overrightarrow{1}, \overrightarrow{a_{i+1}},
\cdots \overrightarrow{a_r}]_{\pi}$.

Now using the filtration $\Fil^{k-1} \subset \Fil^k \subset \Psi_\varrho$, to conclude it suffices to
look at $H^{i_0-1}(X_{\IC,\bar s_{w_0}},\Psi_{\IC,\varrho}/\Fil^k)_{\mathfrak m}
\otimes_{\overline \Zm_l} \overline \Fm_l$. For the Harris-Taylor perverse sheaves
$\PC(\pi'_{w_0},t)(\frac{1-t+2k}{2})$ with $t'>t$ we argue as before, and for the others we invoke
the minimality of $t$ and $i_0$.
\end{proof}

\subsection{From Ihara's lemma to the cohomology}

Recall first that 
$$X_{\IC,\bar s_{v_0}}^{=d}=\coprod_{i \in  \ker^1(\Qm,G)} X_{\IC,\bar s_{v_0},i}^{=d},$$
and that for a $G(\Am^\oo)$-equivariant sheaf $\FC_{\IC,i}$ on $X_{\IC,\bar s_{v_0},i}^{=d}$,
its fiber at some compatible system $z_{i,\IC}$ of supersingular points, has an action of
$\overline G(\Qm) \times GL_d(F_{v_0})^0$ where $GL_d(F_{v_0})^0$ 
is the kernel of the valuation of the determinant so that, cf. \cite{boyer-invent2} proposition 5.1.1, 
as a $G(\Am^\oo)\simeq \overline G(\Am^{\oo,v_0}) \times GL_d(F_{v_0})$-module, we have
$$H^0(\bar X_{\IC,\bar s_{v_0},i}^{=d},\FC_{\IC,i}) \simeq 
\ind_{\overline G(\Qm)}^{\overline G(\Am^{\oo,v_0}) \times \Zm} z_i^* \FC_{\IC,i}$$
with $\delta \in \overline G(\Qm) \mapsto (\delta^{\oo,v_0},\val \circ \rn( \delta_{v_0})) 
\in \overline G(\Am^{\oo,v_0}) \times \Zm$ and where the action of
$g_{v_0} \in GL_d(F_{v_0})$ is given by those of 
$(g_0^{-\val \det g_{v_0}} g_{v_0},\val \det g_{v_0}) \in GL_d(F_{v_0})^0 \times \Zm$
where $g_0 \in GL_d(F_{v_0})$ is any fixed element with $\val \det g_0=1$. 
Moreover, cf. \cite{boyer-invent2} corollaire 5.1.2, if $z_i^* \FC_\IC$ is provided with an action
of the kernel $(D_{v_0,d}^\times)^0$ of the valuation of the reduced norm, action compatible with those
of $\overline G(\Qm) \hookrightarrow D_{v_0,d}^\times$, then as a $G(\Am^\oo)$-module, we have
\addtocounter{smfthm}{1}
\begin{equation} \label{h0-ss}
H^0(X_{\IC,\bar s_{v_0},i}^{=d},\FC_{\IC,i}) \simeq \CC^\oo(\overline G(\Qm) \backslash 
\overline G(\Am^{\oo}),\Lambda)
\otimes_{D_{v_0,d}^\times}  \ind_{(D_{v_0,d}^\times)^0}^{D_{v_0,d}^\times} z_i^* \FC_{\IC,i}
\end{equation}

\begin{lemm} 
Le $\overline \pi$ be an irreducible sub-$\overline \Fm_l$-representation of
$\mathcal C^\oo(\overline G(\Qm) \backslash \overline G(\Am)/U^{v_0},\overline \Fm_l)$.
Denote its local component $\bar \pi_{v_0}$ at $v_0$ as $\pi_{v_0}[s]_D$ with $\pi_{v_0}$
an irreducible cuspidal representation of $GL_g(F_{v_0})$ with $d=sg$. Then
$\overline \pi^{v_0}$ is a sub-representation of 
$H^0(X^{=d}_{U^{v_0},\bar s_{v_0}},HT(\pi_{v_0}^\vee,s)) \otimes_{\overline \Zm_l} \overline \Fm_l$.
\end{lemm}

\begin{proof}
Clearly we have $\overline \pi^{v_0} \subset \mathcal C^\oo(\overline G(\Qm) \backslash 
\overline G(\Am)/U^{v_0},\overline \Fm_l) \otimes \bar \pi_{v_0}^\vee$. The result then follows
from (\ref{h0-ss}) and the definition of the Harris-Taylor local system $HT(\pi_{v_0}^\vee,s)$ 
with support on the supersingular stratum.
\end{proof}

\begin{prop} \label{prop-coho}
Let  $\mathfrak m$ be a maximal ideal of $\Tm_S$ verifying (H1) and (H3),
and let $\bar \pi$ be an irreducible sub-$\overline \Fm_l$-representation of
$\mathcal C^\oo(\overline G(\Qm) \backslash \overline G(\Am)/U^v,\overline \Fm_l)_{\mathfrak m}$.
Then $\bar \pi^{\oo,v}$ is a sub-$\overline \Fm_l$-representation of 
$H^{d-1}(X_{U,\bar \eta_{v_0}},\overline \Fm_l)_{\mathfrak m}$.
\end{prop}

\begin{proof}
By \cite{dat-jl} theorem 3.1.4, $\bar \pi_{v_0}$ is associated, through modulo $l$ Jacquet-Langlands
correspondence, to some superSpeh, $\speh_s(\varrho)$ with $\varrho$ an
irreducible supercuspidal representation of $GL_g(F_{v_0})$ with $d=sg$.
Recall that
$H^i(X_{U,\bar s_{v_0}},\Psi_{\varrho})_{\mathfrak m}$ is a direct factor of 
$H^{d-1}(X_{U,\bar \eta_{v_0}},\overline \Fm_l)_{\mathfrak m}$, so that it suffices to prove that
$\bar \pi^{\oo,v}$ is a sub-$\overline \Fm_l$-representation of 
$H^i(X_{U,\bar s_{v_0}},\Psi_{\varrho})_{\mathfrak m}$.

As in the proof of theorem \ref{theo-torsion} using now (H3),
consider then the filtration of $\Psi_\varrho$ introduced before so that its graduates are some
Harris-Taylor perverse sheaves of type $\varrho$ and whose $\mathfrak m$-localized
cohomology groups are free concentrated
in degree $0$. Note in particular that $\PC(\pi^\vee_{v_0},s)(\frac{s-1}{2})$
is its first graduate so that, using the spectral sequence computing 
$H^i(X_{U,\bar s_{v_0}},\Psi_{\rho})_{\mathfrak m}$ with $E_1^{p,q}$ given by 
the $H^i(X_{U,\bar s_{v_0}},\PC(\pi'_{v_0},t)(\frac{1-t+2k}{2}))_{\mathfrak m}$, we see that
$$H^i(X_{U,\bar s_{v_0}},\PC(\pi^\vee_{v_0},s)(\frac{s-1}{2}))_{\mathfrak m} \harrow 
H^i(X_{U,\bar s_{v_0}},\Psi_{\rho})_{\mathfrak m}$$
with free cokernel, so that 
$H^0(X^{=d}_{U^{v_0},\bar s_{v_0}},\PC(\pi^\vee_{v_0},s)(\frac{s-1}{2}))_{\mathfrak m} $
is a subspace of 
$H^{d-1}(X_{U,\bar \eta_{v_0}},\overline \Fm_l)_{\mathfrak m}$. The result then follows from
the previous lemma.
\end{proof}

The strategy now to prove Ihara's lemma, under our restrictive hypothesis on $\mathfrak m$,
is then to prove the same statement on $H^{d-1}(X_{U,\bar \eta_{v_0}},\overline \Fm_l)_{\mathfrak m}$,
i.e. if $\pi^{\oo,v_0}$ is a subspace of it, then its local component $\pi^{\oo,v_0}_{w_0}$ at the place
$w_0$, is generic. Finally our Ihara's lemma statement will follows from proposition \ref{prop-principale}.

\section{Non degeneracy property for global cohomology}

\subsection{Global lattices are tensorial product}
\label{para-reseau}

From now on we suppose that $\overline \rho_{\mathfrak m}$ is absolutely 
irreducible. 

\begin{theo} (cf. \cite{scholze-LT} theorem 5.6) \\
As a $\Tm_{S,\mathfrak m}[\gal_{F,S}]$-module,
$$H^{d-1}(X_{U,\bar \eta},\Zm_l)_{\mathfrak m} \simeq \sigma_{\mathfrak m} 
\otimes_{\Tm_{S,\mathfrak m}} \rho_{\mathfrak m},$$
for some $\Tm_{S,\mathfrak m}$-module $\sigma_{\mathfrak m}$ on wich $\gal_F$
acts trivially.
\end{theo}

\rem 
The proof is exactly the same than those of theorem 5.6 of \cite{scholze-LT}
so that we are reduced to prove that $H^{d-1}(X_{U,\bar \eta},\overline \Zm_l)_{\mathfrak m}[1/l]$ is $\rho_{\mathfrak m}$-typic in the sense of definition 5.3 of \cite{scholze-LT}
which follows from \cite{h-t} or \cite{boyer-compositio}.
Note moreover that such result for $d=2$ goes back to Carayol in \cite{carayol-local}.

In particular we have the following property.

\begin{prop} \label{prop-tensoriel}
Let $\Pi^{\oo,U} \otimes L_g(\Pi_{v_1}^\vee)$ be a direct factor of  
$H^{d-1}(X_{U,\bar \eta_{v_1}},\overline \Qm_l)_{\mathfrak m}$, and consider its lattice 
given by the $\overline \Zm_l$-cohomology. Then this lattice is a tensorial product
$\Gamma_G \otimes \Gamma_W$ of a stable lattice $\Gamma_G$ (resp. $\Gamma_W$) of
$\Pi^{\oo,U}$ (resp. of $L_d(\Pi_{v_1}^\vee)$).
\end{prop}

\subsection{Proof of the main result}
\label{para-proof}

Let $\SC(\mathfrak m)$ be the supercuspidal support of the modulo $l$ reduction
of any $\Pi_{\widetilde m,w_0}$ in the near equivalence class associated to a minimal prime ideal
$\widetilde{\mathfrak m} \subset \mathfrak m$. Recall that $\SC(\mathfrak m)$ depends only on 
$\mathfrak m$ and by (H2) it's multiplicity free, we decompose it according to the inertial equivalence
classes $\varrho$ of irreducible $\overline \Fm_l$-representations of some $GL_{g(\varrho)}(F_{w_0})$
with $1 \leq g(\varrho) \leq d$:
$$\SC(\mathfrak m)=\coprod_{\varrho} \SC_{\varrho}(\mathfrak m).$$
For each such $\varrho$ we denote
$l_1(\varrho) \geq \cdots \geq l_{r(\varrho)}(\varrho)>0$ integers
so that $S_{\varrho}(\mathfrak m)$ can be written as a disjoint union of $r(\varrho)$ 
unlinked Zelevinsky segments
$$[\varrho\{ \delta_i \},\varrho \{ \delta_i+l_i(\varrho) \}]=\bigl \{ \varrho \{\delta_i \},
\varrho \{\delta_i+1 \},\cdots, \varrho \{\delta_i+l_i(\varrho) \} \bigr \}.$$
An irreducible $\overline \Fm_l$-representation $\tau_{w_0}$ of $GL_d(F_{w_0})$ whose supercuspidal
support is equal to $\SC(\mathfrak m)$, can be written as a full induced $\tau_{w_0} \simeq \bigtimes_{\varrho}
\tau_{\varrho}$ where each $\tau_{\varrho}$ is also a full induced representation
$$\tau_\varrho \simeq \bigtimes_{i=1}^{r(\varrho)} \tau_{\varrho,i}$$
with $\tau_{\varrho,i}$ of supercuspidal support equals to those of 
$[\varrho \{\delta_i \},\varrho \{ \delta_i+l_i(\varrho) \}]$. With the notations of \ref{nota-segment-modl},
each of these $\tau_{\varrho,i}$ can be written as
$$\tau_{\varrho,i} \simeq [\overleftarrow{a_1(\varrho)},\overrightarrow{a_2(\varrho)},\cdots,
\overrightarrow{a_{t_i(\varrho)}(\varrho)}]_\varrho.$$

\begin{defi} We say, cf. the remark following notation \ref{nota-segment-modl},
that $\tau_{w_0}$ is non degenerate if for all $\varrho$ and for all $1 \leq i \leq r(\varrho)$,
then $\tau_{\varrho,i} \simeq [\overleftarrow{a_i(\varrho)}]_\varrho$.
\end{defi}

\begin{prop} \label{prop-principale}
Let $\tau_{w_0}$ be an irreducible representation of $GL_d(F_{w_0})$ which is a subspace of
$$H^{d-1}(X_{U^{w_0}(\oo),\bar \eta_{w_0}},\overline \Fm_l)_{\mathfrak m}:=
\lim_{\atop{\rightarrow}{n}} H^{d-1}(X_{U^{w_0}(n),\bar \eta_{w_0}},\overline \Fm_l)_{\mathfrak m}.$$
Then $\pi_{w_0}$ is non degenerate.
\end{prop}

\begin{proof}
Note first that the supercuspidal support of $\tau_{w_0}$ must be $\SC(\mathfrak m)$.
The exhaustive filtration of $\Psi_{\varrho_0}$, cf. \S \ref{para-entier}, whose graduates are
Harris-Taylor perverse sheaves, gives a filtration of 
$H^{d-1}(X_{U^{w_0}(\oo),\bar \eta_{w_0}},\overline \Fm_l)_{\mathfrak m}$,
whose graduate are, thanks to theorem \ref{theo-torsion}, the
$$H^{0}(X_{U^{w_0}(\oo),\bar \eta_{w_0}},\PC(\pi_{w_0},t)(\frac{1-t+2k}{2}))_{\mathfrak m} 
\otimes_{\overline \Zm_l} \overline \Fm_l$$ 
for $\pi_{w_0} \in \scusp_{-1}(\varrho)$ with $\varrho$ such that $S_\varrho(\mathfrak m)$ is
non empty. Then $\tau_{w_0}$ must be a subspace of one of these graduates.
We argue by absurdity using the following lemma.

\begin{lemm} \label{lem-ss}
If $\rho$ be a subspace of $[\overleftarrow{t-1}]_{\varrho \{\frac{-\delta}{2} \} }\times \rho'$ then
with the notation \ref{nota-segment-modl}, 
\begin{itemize}
\item if $\delta=s-t$ then $\rho=
[\overleftarrow{t-1},\overrightarrow{1},\overbrace{\cdots}^{s-t-1}]_{\varrho}$;

\item if $\delta=t-s$ then $\rho=[\overbrace{\cdots}^{s-t-1},\overleftarrow{1},\overleftarrow{t-1}]_{\varrho}$;

\item otherwise, i.e. $t-s< \delta<s-t$, then 
$$\rho=[\overbrace{\cdots}^{s-t-\delta-1},\overleftarrow{1},\overleftarrow{t-1},\overrightarrow{1},
\overbrace{\cdots}^{s-t+\delta-1}]_\varrho.$$
\end{itemize}
\end{lemm}

\begin{proof}
The result is well known over $\overline \Qm_l$ and we can easily argue in the same way using
\begin{itemize}
\item the fact that all the $\varrho\{ \frac{1-s}{2}+k \}$ for $0 \leq k \leq s-1$ are pairwise distinct

\item and the property of commutation between the modulo $l$ reduction and the Jacquet functors.
\end{itemize}
Consider for example the case $t-s< \delta<s-t$.
By Frobenius reciprocity we see that the subspace we look for, is some
undetermined irreducible subspace of the modulo $l$ reduction of 
$[\overbrace{\cdots}^{s-t-\delta-1},\overleftarrow{1},\overleftarrow{t-1},\overrightarrow{1},
\overbrace{\cdots}^{s-t+\delta-1}]_\pi$. By convention, cf. notation
\ref{nota-segment-modl}, we denote such a subquotient 
$[\overbrace{\cdots}^{s-t-\delta-1},\overleftarrow{1},\overleftarrow{t-1},\overrightarrow{1},
\overbrace{\cdots}^{s-t+\delta-1}]_\varrho$.
\end{proof}

Suppose now, by absurdity, it exists an irreducible supercuspidal $\overline \Fm_l$-representation
$\varrho_0$ such that $\tau_{\varrho_0}$ is degenerate and take $i$ with
$$\tau_{\varrho_0,i} \simeq [\cdots,\overrightarrow{a},\cdots]_{\varrho_0},$$
and let $\beta \in \frac{1}{2} \Zm$ such that $\varrho_0 \{ \beta \}$ is the supercuspidal 
corresponding to the end of the arrow $\overrightarrow{a}$ in the previous notation.
From proposition \ref{prop-tensoriel}, we see that 
$\tau_{w_0} \otimes \overline \rho_{\mathfrak m}$ is a 
$\overline \Fm_l[GL_d(F_{w_0}) \times \gal(\overline F/F)]$-submodule of
$H^{d-1}(X_{U^{w_0}(\oo),\bar \eta_{w_0}},\overline \Fm_l)_{\mathfrak m}$.
After restricting the Galois action to the Weil group at $w_0$, we see that
$\tau_{w_0} \otimes \Lm(\varrho_0 \{ \beta \} )$ has to be a  of 
$\overline \Fm_l[GL_d(F_{w_0}) \times W_{w_0}]$-submodule of
$H^{d-1}(X_{U^{w_0}(\oo),\bar \eta_{w_0}},\overline \Fm_l)_{\mathfrak m}$
and, as before, on one of the
$$H^{0}(X_{U^{w_0}(\oo),\bar \eta_{w_0}},\PC(\pi_{w_0},t)(\frac{1-t+2k}{2}))_{\mathfrak m} 
\otimes_{\overline \Zm_l} \overline \Fm_l$$ 
for $\pi_{w_0} \in \scusp_{-1}(\varrho_0)$. 
Recall that this last cohomology group is parabolically induced from
$$H^{0}(X^{\geq tg}_{U^{w_0}(\oo),\bar \eta_{w_0},\overline{1_{tg}}},
\PC_1(\pi_{w_0},t)(\frac{1-t+2k}{2}))_{\mathfrak m} 
\otimes_{\overline \Zm_l} \overline \Fm_l$$ 
where by lemma \ref{lem-coho-flou}, every irreducible
$\overline \Fm_l[P_{tg,d}(F_{w_0}) \times W_{w_0}]$-subquotient 
of it can be written as $[\overleftarrow{t-1}]_{\varrho_0 \{-\frac{\delta}{2}\} } \otimes \tau \otimes
\Lm(\varrho_0 \{ \alpha \})$ where $\tau$ is any irreducible representation of $GL_{d-tg}(F_{w_0})$
and $\alpha \in \frac{1}{2} \Zm$ is such that $\varrho_0 \{\alpha\}$ belongs to the supercuspidal
support of $[\overleftarrow{t-1}]_{\varrho_0 \{ -\frac{\delta}{2} \} }$.

The contradiction then follows from the previous lemma.
\end{proof}

Finally our restricted version of Ihara's lemma given in the introduction, follows 
from propositions \ref{prop-coho} and \ref{prop-principale}.

\rem Note that in the previous proof we used the second part of (H2) to say that the modulo $l$ reduction of 
$[\overleftarrow{s-1}]_\pi$ is irreducible and so any of its subspace is non degenerate, cf.
the remark just before theorem \ref{theo-torsion}. Using the main
result of \cite{boyer-local-ihara}, we have this last property without any hypothesis so, as this is the only place
where we use the second part of (H2), we can remove it.

\subsection{Level rising}

Before dealing with the general case, consider first the case $d=2$ and take $l \geq 3$ such that the order of 
$q_{w_0}$ modulo $l$ is $2$. Suppose then, by absurdity, there exists a maximal 
ideal $\mathfrak m$ such that 
\begin{itemize}
\item[(a)] for every prime ideal $\widetilde{\mathfrak m} \subset \mathfrak m$, the local component at
$w_0$ of $\Pi_{\widetilde{\mathfrak m}}$ is unramified;

\item[(b)] for such prime ideal, write $\Pi_{\widetilde{\mathfrak m},w_0} \simeq 
\chi_{w_0,1} \times \chi_{w_0,2}$, and suppose $\chi_{w_0,1} \chi_{w_0,2}^{-1} \equiv \nu \mod l$.
\end{itemize}
Using (a) and the spectral sequence of vanishing cycles at $w_0$, we obtain
$$H^1(X_U,\overline \Fm_l)_{\mathfrak m} \simeq
H^1(X_{U,\bar s_{w_0}}^{=1},\Psi(\overline \Fm_l))_{\mathfrak m}$$
where $X_{U,\bar s_{w_0}}^{=1}$ is the ordinary locus of the geometric special fiber of $X_U$ at $w_0$.
It's well known that this cohomology group is parabolic induced. Moreover the only non degenerate
irreducible representation of $GL_d(F_{w_0})$ which is a subquotient of the modulo $l$ reduction of 
$\chi_{w_0,1} \times \chi_{w_0,1} \nu$ is cuspidal, because of the fact that $q_{w_0}$ is of order $2$
modulo $l$, this non degenerate representation can't be a subspace of the induced representation
$H^1(X_U,\overline \Fm_l)_{\mathfrak m}$. The contradiction then given by Ihara's lemma.

In higher dimension, recall first the notations of the beginning of the previous section.
For a minimal prime ideal $\widetilde{\mathfrak m} \subset \mathfrak m$ and an automorphic
representation $\Pi \in \Pi_{\widetilde{\mathfrak m}}$ in the near equivalence class associated to
$\widetilde{\mathfrak m}$, we write its local component at $w_0$
$$\Pi_{w_0} \simeq \bigtimes_{\varrho} \Pi_{w_0}(\varrho)$$
and $\Pi_{w_0}(\varrho) \simeq \bigtimes_{i=1}^{r(\varrho)} \Pi_{w_0}(\varrho,i)$
where for each $1 \leq i \leq r(\varrho)$, the modulo $l$ reduction of the supercuspidal support
of $\Pi_{w_0}(\varrho,i)$ is, with the notations of the previous section, those of the Zelevinsky segment 
$[\varrho \{\delta_i \},\varrho \{\delta_i+l_i(\varrho) \} ]$.

\begin{prop} \label{prop-diminution}
Take a maximal ideal $\mathfrak m$ verifying the hypothesis (H1) and (H2).
Let $\varrho_0$ such that $\SC_{\varrho_0}(\mathfrak m)$ is non empty and consider 
$1 \leq i \leq r(\varrho_0)$. Then there exists a minimal prime ideal 
$\widetilde{\mathfrak m} \subset \mathfrak m$ and an automorphic
representation $\Pi \in \Pi_{\widetilde{\mathfrak m}}$ such that with the previous notation
$\Pi_{w_0}(\varrho_0,i)$ is non degenerate, i.e. isomorphic to $\st_{l_i(\varrho)}(\pi_{w_0})$
for some irreducible cuspidal $\overline \Qm_l$-representation $\pi_{w_0}$.
\end{prop}

\rem In particular if $\SC(\mathfrak m)=\SC_{\varrho_0}(\mathfrak m)$ and $r(\varrho_0)=1$, i.e.
the supercuspidal support of the modulo $l$ reduction of the local component at $w_0$ of any
$\Pi \in \Pi_{\widetilde{\mathfrak m}}$ for any $\widetilde{\mathfrak m} \subset \mathfrak m$,
is a Zelevinsky segment, then $\Pi_{w_0}$ is non degenerate. This is the case considered
in section 4.5 of \cite{CHT}. In an incoming work, we intend to explain how to rise the level
 simultaneously for all $1 \leq i \leq r(\varrho_0)$ and all $\varrho_0$ together.

\begin{proof}
For a minimal prime ideal $\widetilde{\mathfrak m} \subset \mathfrak m$
and $\Pi \in \Pi_{\widetilde{\mathfrak m}}$, we write
$$\Pi_{w_0}(\varrho_0,i) \simeq \st_{s_1}(\pi_{w_0,1}) \times \cdots \st_{s_a}(\pi_{w_0,a})$$
where $s_1 \geq s_2 \geq \cdots \geq s_a \geq 1$ and $\pi_{w_0,1},\cdots,\pi_{w_0,a}$
irreducible cuspidal $\overline \Qm_l$-representations of type $\varrho_0$ of $GL_{g_i}(F_{w_0})$.
We then argue by absurdity, i.e. we suppose $a \geq 2$ for all $\widetilde{\mathfrak m} \subset \mathfrak m$
and we choose such $\widetilde{\mathfrak m}$ so that $s_1$ is maximal.
The strategy is then, using lemma \ref{lem-ss}, to construct a degenerate 
$\overline \Fm_l[GL_d(F_{w_0})]$-subspace of 
$H^{d-1}(X_{U^{w_0}(\oo),\bar \eta_{w_0}},\overline \Fm_l)_{\mathfrak m}$ which contradicts 
the genericness property of irreducible sub-modules of this cohomology group as proved before.
In \cite{boyer-compositio} \S 3.6, we prove that for all minimal prime 
$\widetilde{\mathfrak m}' \subset \mathfrak m$ 
$$H^i(X_{U,\bar s_{w_0}},HT_{\overline \Qm_l}(\pi_{w_0,1},t))_{\widetilde{\mathfrak m}'}=(0)$$
 either if $t>s_1$ or for $t=s_1$, if $i \neq 0$.
Consider now the filtration 
$$\Fil^{- s_1g_1}_*(\Psi_{\varrho_0}) \harrow \Fil^{1- s_1g_1}_*(\Psi_{\varrho_0}) 
\harrow \Psi_{\varrho_0}$$
and recall that, by construction, 
\begin{itemize}
\item $\Fil^{- s_1g_1}_*(\Psi_{\varrho_0})$ is supported in
$X^{> s_1g_1}_{\IC,\bar s_{w_0}}$ 

\item $\gr^{1- s_1g_1}_*(\Psi_{\varrho_0}) \simeq 
\bigoplus_{\pi_{w_0} \in \scusp_{-1}(\varrho_0)} \PC(\pi_{w_0},s_1)(\frac{s_1-1}{2})$.
\end{itemize}
By the theorem \ref{theo-torsion} we know the cohomology groups of Harris-Taylor perverse
sheaves to be free, so
\begin{itemize}
\item $H^i(X_{U,\bar s_{w_0}},\Fil^{- s_1g_1}_*(\Psi_{\varrho_0}))_{\mathfrak m}=(0)$; 

\item $H^i(X_{U,\bar s_{w_0}},\Psi_{\varrho_0}/\Fil^{- s_1g_1}_*(\Psi_{\varrho_0}))_{\mathfrak m}$
is free.
\end{itemize}
Recall moreover, cf. \cite{boyer-compositio} \S 3.6, that $\Pi_{w_0} \otimes \Lm(\pi_{w_0,1})(\frac{s_1-1}{2})$
is a direct factor of 
$$H^i(X_{U^{w_0}(\oo),\bar s_{w_0}},
HT_{\overline \Qm_l}(\pi_{w_0,1},s_1)(\frac{s_1-1}{2}))_{\widetilde{\mathfrak m}},.$$
Moreover the stable lattice given by the $\overline \Zm_l$-cohomology looks like
$\bigl ( \Gamma(\varrho_0,1) \times \Gamma^{\varrho_0,1} \bigr ) \times \Gamma_W$
where 
\begin{itemize}
\item $\Gamma(\varrho_0,1)$ is a stable lattice of $\st_{s_1}(\pi_{w_0,1})$,

\item $\Gamma^{\varrho_0,1}$ is a stable lattice of $\bigl (\bigtimes_{\varrho \neq \varrho_0} 
\Pi_{w_0}(\varrho) \bigr ) \times \bigl ( \bigtimes_{i=2}^{r(\varrho_0)} \Pi_{w_0}(\varrho_0,i) \bigr )$,

\item and $\Gamma_W$ is a stable lattice of $\Lm(\pi_{w_0,1})(\frac{s_1-1}{2})$.
\end{itemize}
The result then follows from lemma \ref{lem-ss}.

\end{proof}

\bibliographystyle{plain}
\bibliography{bib-ok}

\def\cftil#1{\ifmmode\setbox7\hbox{$\accent"5E#1$}\else
  \setbox7\hbox{\accent"5E#1}\penalty 10000\relax\fi\raise 1\ht7
  \hbox{\lower1.15ex\hbox to 1\wd7{\hss\accent"7E\hss}}\penalty 10000
  \hskip-1\wd7\penalty 10000\box7} \def\cprime{$'$}
\begin{thebibliography}{10}

\bibitem{allen-newton}
P.~B. Allen and J.~Newton.
\newblock Monodromy for some rank two galois representations over {CM} fields.
\newblock {\em preprint}, 2019.

\bibitem{ast}
A.~A. Beilinson, J.~Bernstein, and P.~Deligne.
\newblock Faisceaux pervers.
\newblock In {\em Analysis and topology on singular spaces, I (Luminy, 1981)},
  volume 100 of {\em Ast\'erisque}, pages 5--171. Soc. Math. France, Paris,
  1982.

\bibitem{boyer-invent2}
P.~Boyer.
\newblock Monodromie du faisceau pervers des cycles \'evanescents de quelques
  vari\'et\'es de {S}himura simples.
\newblock {\em Invent. Math.}, 177(2):239--280, 2009.

\bibitem{boyer-compositio}
P.~Boyer.
\newblock Cohomologie des syst\`emes locaux de {H}arris-{T}aylor et
  applications.
\newblock {\em Compositio}, 146(2):367--403, 2010.

\bibitem{boyer-repmodl}
P.~Boyer.
\newblock {R}\'eseaux d'induction des repr\'esentations elliptiques de
  {L}ubin-{T}ate.
\newblock {\em Journal of Algebra}, 336, issue 1:28--52, 2011.

\bibitem{boyer-torsion}
P.~Boyer.
\newblock Filtrations de stratification de quelques vari\'et\'es de {S}himura
  simples.
\newblock {\em Bulletin de la SMF}, 142, fascicule 4:777--814, 2014.

\bibitem{boyer-aif}
P.~Boyer.
\newblock Congruences automorphes et torsion dans la cohomologie d'un syst\`eme
  local d'{H}arris-{T}aylor.
\newblock {\em Annales de l'Institut Fourier}, 65 no. 4:1669--1710, 2015.

\bibitem{boyer-imj}
P.~Boyer.
\newblock Sur la torsion dans la cohomologie des vari\'et\'es de {Sh}imura de
  {K}ottwitz-{H}arris-{T}aylor.
\newblock {\em Journal of the Institute of Mathematics of Jussieu}, pages
  1--19, 2017.

\bibitem{boyer-local-ihara}
P.~Boyer.
\newblock Persitence of non degeneracy: a local analog of {I}hara's lemma.
\newblock {\em soumis}, pages 1--39, 2018.

\bibitem{boyer-bordeaux}
P.~Boyer.
\newblock About {G}alois reductibility of torsion cohomology classes for {KHT}
  {S}himura varieties.
\newblock {\em https://www.math.univ-paris13.fr/$\sim$
  boyer/recherche/rho-irred.pdf}, 2019.

\bibitem{boyer-FT}
P.~Boyer.
\newblock Groupe mirabolique, stratification de {N}ewton raffin\'ee et
  cohomologie des espaces de {L}ubin-{T}ate.
\newblock {\em Bull. SMF}, pages 1--18, 2019.

\bibitem{scholze-cara}
Ana Caraiani and Peter Scholze.
\newblock On the generic part of the cohomology of compact unitary {S}himura
  varieties.
\newblock {\em Ann. of Math. (2)}, 186(3):649--766, 2017.

\bibitem{carayol-local}
H.~Carayol.
\newblock Formes modulaires et repr\'esentations galoisiennes \`a valeurs dans
  un anneau local complet.
\newblock {\em $p$-adci monodromy and the {B}irch and {S}winnerton-{D}yer
  conjecture (Boston, MA, 1991)}, 165 of Cotem. Math.:213--237, 1994.

\bibitem{CHT}
L.~Clozel, M.~Harris, and R.~Taylor.
\newblock Automorphy for some $l$-adic lifts of automorphic mod $l$
  representations.
\newblock {\em preprint, http://www.math.harvard.edu/$\sim$rtaylor/}, 2004.

\bibitem{clo-thor}
L.~Clozel and J.~Thorne.
\newblock Level-raising and symmetric power functoriality, i.
\newblock {\em Compositio Mathematica}, 150, No. 5:729--748, 2014.

\bibitem{dat-torsion}
J.-F. Dat.
\newblock Th\'eorie de {L}ubin-{T}ate non-ab\'elienne $l$-enti\`ere.
\newblock {\em Duke Math. J. 161 (6)}, pages 951--1010, 2012.

\bibitem{dat-jl}
J.-F. Dat.
\newblock Un cas simple de correspondance de {J}acquet-{L}anglands modulo $l$.
\newblock {\em Proc. London Math. Soc. 104}, pages 690--727, 2012.

\bibitem{emer-helm}
M.~Emerton and D.~Helm.
\newblock The local langlands correspondence for {G}l(n) in families.
\newblock {\em Ann. Sci. E.N.S. (4)}, 47:655--722, 2014.

\bibitem{h-t}
M.~Harris, R.~Taylor.
\newblock {\em The geometry and cohomology of some simple {S}himura varieties},
  volume 151 of {\em Annals of Mathematics Studies}.
\newblock Princeton University Press, Princeton, NJ, 2001.

\bibitem{ill}
L.~Illusie.
\newblock Autour du th\'eor\`eme de monodromie locale.
\newblock In {\em P\'eriodes $p$-adiques}, number 223 in Ast\'erisque, 1994.

\bibitem{scholze-LT}
Peter Scholze.
\newblock On the p-adic cohomology of the {L}ubin-{T}ate tower.
\newblock {\em Ann. Sci. …c. Norm. Sup\'er.}, (4) 51:no. 4, 811--863., 2018.

\bibitem{y-t}
R.~Taylor and T.~Yoshida.
\newblock Compatibility of local and global {L}anglands correspondences.
\newblock {\em J.A.M.S.}, 20:467--493, 2007.

\bibitem{vigneras-induced}
M.-F. Vign{\'e}ras.
\newblock Induced {$R$}-representations of {$p$}-adic reductive groups.
\newblock {\em Selecta Math. (N.S.)}, 4(4):549--623, 1998.

\bibitem{zelevinski2}
A.~V. Zelevinsky.
\newblock Induced representations of reductive {${p}$}-adic groups. {II}. {O}n
  irreducible representations of {${\rm GL}(n)$}.
\newblock {\em Ann. Sci. \'Ecole Norm. Sup. (4)}, 13(2):165--210, 1980.

\end{thebibliography}

\end{document}